\RequirePackage{tikz}
\documentclass[a4,10pt]{article}

\usepackage{amssymb,amsmath}

\usepackage{amsthm}
\usepackage{bm}
\usepackage{environ} 
\usepackage{todonotes}
\usepackage{booktabs}
\usepackage{csquotes}
\usepackage{mathtools}
\usepackage{subcaption}
\usepackage{picins}
\usepackage{enumitem}
\setenumerate{label=\alph*)}
\usepackage[percent]{overpic}
\usepackage[pdftex,colorlinks=true]{hyperref}
	\usepackage{authblk}
	\usepackage[margin=1.2in]{geometry}
	
\theoremstyle{definition}
\newtheorem{thm}{Theorem}[section]
\newtheorem{defn}[thm]{Definition}
\newtheorem{lem}[thm]{Lemma}
\newtheorem{prop}[thm]{Proposition}
\newtheorem{rem}[thm]{Remark}
\newtheorem{cor}[thm]{Corollary}

\newcommand{\mbfz}{\mathbf z}
\newcommand{\bbf}{\mathbf f}

\newcommand{\mbfy}{\mathbf y}

\newcommand{\bh}{\mathbf h}

\newcommand{\bg}{\mathbf g}
\newcommand{\bA}{\mathbf A}
\newcommand{\bC}{\mathbf C}
\newcommand{\bM}{\mathbf M}

\newcommand{\bB}{\mathbf B}
\newcommand{\bS}{\mathbf S}
\newcommand{\bT}{\mathbf T}

\newcommand{\bD}{\mathbf D}
\newcommand{\bH}{\mathbf H}

\newcommand{\bR}{\mathbf Q_{\by^*}}

\newcommand{\bI}{\mathbf I}

\newcommand{\bG}{\mathbf G}

\newcommand{\bW}{\mathbf W}
\newcommand{\bzero}{\mathbf 0}
\newcommand{\bv}{\mathbf v}
\newcommand{\by}{\mathbf y}

\newcommand{\bP}{\mathbf P}
\newcommand{\bw}{\mathbf w}
\newcommand{\bU}{\mathbf U}
\newcommand{\bV}{\mathbf V}
\newcommand{\bu}{\mathbf u}

\newcommand{\bssigma}{\boldsymbol{\sigma}}
\newcommand{\btau}{\boldsymbol{\tau}}

\newcommand{\R}{\mathbb R}
\newcommand{\C}{\mathbb C}

\newcommand{\abs}[1]{\lvert #1\rvert}
\newcommand{\norm}[1]{\lVert #1\rVert}

\newcommand{\be}{\mathbf e}

\newcommand{\qta}{\quad\text{ and }\quad}

\newcommand{\N}{\mathbb{N}}

\newcommand{\bby}{\bar{\by}}
\renewcommand{\Vec}[1]{\renewcommand*{\arraystretch}{1.2}\begin{pmatrix*}[r]#1\end{pmatrix*}}

\newcommand{\Mat}[1]{\Vec{#1}}
\newcommand{\from}{\colon}

\DeclareMathOperator{\diag}{diag}
\DeclareMathOperator{\tr}{trace}

\makeatletter
\newcommand{\vast}{\bBigg@{3}}
\newcommand{\Vast}{\bBigg@{4}}
\makeatother
\newcommand{\vastl}{\mathopen\vast}

\newcommand{\vastr}{\mathclose\vast}
\newcommand{\Vastl}{\mathopen\Vast}

\newcommand{\Vastr}{\mathclose\Vast}

\makeatletter
\newif\ifcenter@asb@\center@asb@false
\def\center@arstrutbox{%
    \setbox\@arstrutbox\hbox{$\vcenter{\box\@arstrutbox}$}%
    }
\newcommand*{\CenteredArraystretch}[1]{%
    \ifcenter@asb@\else
      \pretocmd{\@mkpream}{\center@arstrutbox}{}{}%
      \center@asb@true
    \fi
    \renewcommand{\arraystretch}{#1}%
    }
\makeatother

\usepackage{tikz}
\usetikzlibrary{patterns}
\definecolor{colorA}{rgb}{0,0.447,0.741}
\definecolor{colorB}{rgb}{0.85,0.325,0.098}
\definecolor{colorE}{rgb}{0.929,0.694,0.125}
\definecolor{colorF}{rgb}{0.494,0.184,0.556}
\definecolor{colorD}{rgb}{0.466,0.674,0.188}
\definecolor{colorC}{rgb}{0.301,0.745,0.933}
\definecolor{colorG}{rgb}{0.635,0.078,0.184}

\newlength{\tickl}    
\tikzset{axes/.style={thick,-latex}}
\tikzset{lineplot/.style={thick}}
\tikzset{arrow/.style={thick,-latex}} 
\tikzset{thick arrow/.style={ultra thick,-latex}} 
\tikzset{grid lines/.style={very thin,gray!30}}	
\tikzset{point/.style={radius=2pt}}
\tikzset{help line/.style={black,thin,dashed}} 
\newcommand{\xticks}[5][0]{
	\foreach \x in {#2}{
		\draw[thin] ({\x},{#1-#3/2}) -- ({\x},{#1+#3/2}) node[below,pos=0,align=center,#5] {#4};
	}
}
\newcommand{\yticks}[5][0]{
	\foreach \y in {#2}{
		\draw[thin] ({#1-#3/2},{\y}) -- ({#1+#3/2},{\y}) node[left,pos=0,align=center,#5] {#4};
	}
}
\newcommand{\kosgrid}[8][0]{
	\draw[grid lines,step={(#8)}] ({#2+0.1},{#4+0.1}) grid ({#3-0.1},{#5-0.1});
	\draw[axes] (#2,#1) -- (#3,#1) node[right]{#6};
	\draw[axes] (0,#4) -- (0,#5) node[left]{#7};
	}
	
\makeatletter
\newsavebox{\measure@tikzpicture}
\NewEnviron{scaletikzpicturetowidth}[1]{%
  \def\tikz@width{#1}%
  \def\tikzscale{1}\begin{lrbox}{\measure@tikzpicture}%
  \BODY
  \end{lrbox}%
  \pgfmathparse{#1/\wd\measure@tikzpicture}%
  \edef\tikzscale{\pgfmathresult}%
  \BODY
}
\makeatother

\begin{document}
\title{On Lyapunov Stability of Positive and Conservative Time Integrators and Application to Second Order Modified Patankar--Runge--Kutta Schemes}
\author[1]{Thomas Izgin} 
\author[1]{Stefan Kopecz}
\author[1]{Andreas Meister}
\affil[1]{Department of Mathematics, University of Kassel, Germany}
\affil[]{izgin@mathematik.uni-kassel.de\ \&\ kopecz@mathematik.uni-kassel.de\ \&\ meister@mathematik.uni-kassel.de}
\setcounter{Maxaffil}{0}
\renewcommand\Affilfont{\itshape\small}
\date{}
%

%
%
\maketitle
\begin{abstract} Since almost twenty years, modified Patankar--Runge--Kutta (MPRK) methods have proven to be efficient and robust numerical schemes that preserve positivity and conservativity of the production-destruction system irrespectively of the time step size chosen. Due to these advantageous properties they are used for a wide variety of applications. Nevertheless, until now, an analytic investigation of the stability of MPRK schemes is still missing, since the usual approach by means of Dahlquist's equation is not feasible. 
	Therefore, we consider a positive and conservative 2D test problem and provide statements usable for a stability analysis of general positive and conservative time integrator schemes based on the center manifold theory.
	We use this approach to investigate the Lyapunov stability of
	the second order MPRK22($\alpha$) and MPRK22ncs($\alpha$) schemes.
	We prove that MPRK22($\alpha$) schemes are unconditionally stable and derive the stability regions of MPRK22ncs($\alpha$) schemes. Finally, numerical experiments are presented, which confirm the theoretical results.  \end{abstract}
\section{Introduction}\label{sec:intro}

In recent years, there has been a strong interest in the development of numerical schemes that preserve properties of the solutions of differential equations. Modified Patankar--Runge--Kutta (MPRK) methods, see \cite{BDM2003,KM18,KM18Order3,MR3969000,MR3934688,MR4064785}, guarantee positivity and conservativity of the numerical solution of positive and conservative production-destruction systems (PDS). For other recent approaches which facilitate positive and conservative numerical approximations, we refer to \cite{MR4087156,MR4109346, MR4194400,nuesslein2021positivitypreserving,blanes2021positivitypreserving}.

A PDS \[\by'=\bP(\by)-\bD(\by),\quad \by(0)=\by^0,\] with $\by=(y_1,\dots,y_N)^T$ and $\bP,\bD\geq \bzero$, is called positive if $\by^0>\bzero$ implies $\by(t)>\bzero$ for all times $t>0$ and is called conservative if $\sum_{i=1}^N y_i(t)=\sum_{i=1}^N y_i^0$ for all times $t>0$. Conditions which ensure the positivity of a PDS are given in \cite{FormaggiaScotti2011}. A conservative PDS can always be written in the form

\[y_i'=\sum_{j=1}^N (p_{ij}(\by)-d_{ij}{(\by)})\quad\text{with}\quad p_{ij}(\by)=d_{ji}(\by) \qta p_{ij}(\by), d_{ij}(\by) \ge 0\] 
for all $\by \ge \bzero$ and $i,j=1,\dots,N$ in which $p_{ij}(\by)$ refers to a production term of the $ith$ equation with corresponding destruction term $d_{ji}(\by)$ in the $jth$ equation. Analogously, $d_{ij}(\by)$ denotes a destruction term of equation $i$ with associated production term $p_{ji}(\by)$ in equation $j$. 
In summary, the solution of a positive and conservative PDS remains positive for all times $t>0$ and the sum of the solutions components remains constant for all times $t>0$.

Originally introduced in \cite{BDM2003}, there has been a considerable interest in the development of MPRK schemes in recent years. In \cite{KM18,KM18Order3,MR3987250} MPRK schemes of second and third order were introduced. These were generalized in the context of SSP Runge--Kutta methods in \cite{MR3969000,MR3934688} and applied to solve reactive Euler equations. In \cite{MR4064785} the idea of \cite{BDM2003} was used to develop mPDeC schemes, which are MPRK schemes of arbitrary order based on deferred correction schemes. All these schemes are unconditionally positive and conservative and have proven their efficiency and robustness while integrating stiff PDS.

MPRK schemes have been used in a wide range of applications. The first order modified Patankar--Euler scheme, introduced in \cite{BDM2003}, is used in a global ocean mercury model with a methylation cycle \cite{SemeniukDastoor2017}.
The second order MPRK scheme of \cite{BDM2003} is applied to an ecosystem model for the simulation of the cyanobacteria life cycle \cite{HenseBeckmann2010,HenseBurchard2010} or that of dinoflagellates \cite{WHK2013}.  
In \cite{BMZ2007,BMZ2009,MeisterBenz2010} this scheme is also used to model the phosphor cycle in rivers and lakes. 
Moreover this scheme was found to be beneficial when applied to NPZD-models in \cite{BDM2005} and is also implemented in the General Ocean Turbulence Model (GOTM) \cite{BBKMNU2006}.
Further applications can be found in the context of magneto-thermal winds \cite{Gressel2017} or warm-hot intergalactic mediums \cite{KlarMuecket2010}. 

 Often MPRK schemes are used within a splitting ansatz as a time integrator for the reactive part of the considered system of partial differential equations in order to avoid additional time step restrictions arising from stiff reaction terms.
 In \cite{CMOT21} mPDeC schemes are used as time integrators for the shallow water equations to ensure unconditionally positivity of the water height.
In \cite{SchippmannBurchard2011} it was demonstrated that the second order MPRK scheme of \cite{BDM2003} surpasses standard Runge--Kutta and Rosenbrock methods for the solution of conservative biochemical models in performance. This was also confirmed in \cite{BonaventuraDellaRocca2016} where the Brusselator PDS was solved with different time integration methods. 
In \cite{WS21} a third order MPRK scheme from \cite{KM18Order3} was successfully used in a high-order operator-splitting method for the numerical solution of the SIR epidemic model.

To the authors knowledge, no general stability analysis of MPRK schemes has been carried out so far.
There are several reasons for the lack of such an analysis. First of all, unlike Runge--Kutta methods, MPRK schemes cannot be applied to the scalar Dahlquist equation \begin{equation}\label{eq:lin_eq}y'=\lambda y,\quad\lambda\in\C^-,\end{equation} since 
it is unclear how to treat the complex term $\lambda y$ in the production-destruction setting.
This issue can be handled by choosing $\lambda\in\R^-$ and considering the system $y_1'=\lambda y_1$, $y_2'=-\lambda y_1$.
For this system the first order MPRK scheme of \cite{BDM2003} is equivalent to the L-stable backward Euler method. Also second order MPRK schemes show an excellent stability behavior when applied to this system, as we show in Section~\ref{b=0_1}. 
Unfortunately, this stability behavior can only be observed for specific MPRK schemes in more general cases, which requires a more detailed analysis.

The scalar Dahlquist equation is so valuable, since it makes a direct stability analysis of the linear system \begin{equation}\label{eq:lin_sys}\by'=\bA\by\end{equation}  with $\bA\in\R^{N\times N}$ unnecessary. A Runge--Kutta method applied to \eqref{eq:lin_sys} has the same stability properties as applied to the $N$ scalar equations \eqref{eq:lin_eq} with $\lambda$ passing through the $N$ eigenvalues of $\bA$, see for instance \cite[Chapter 6]{MR1912409}.

Since the direct application of MPRK schemes to \eqref{eq:lin_eq} is not possible, stability should be investigated for a linear system \eqref{eq:lin_sys}, where we also require the system to be positive and conservative.
In \cite{ortlebhundsdorfer2017} it is pointed out that the inherent nonlinear nature of MPRK schemes makes even a linear stability analysis difficult. To see this, we follow \cite{ortlebhundsdorfer2017} and consider the second order MPRK scheme of \cite{BDM2003} applied to a conservative and positive linear PDS of the form \eqref{eq:lin_sys}. The resulting scheme is given by
\[\by^{(2)}=\by^n+\Delta t\bA \by^{(2)},\quad \by^{n+1}=\by^n+\frac{\Delta t}{2}\bA(\bW^n +\bI)\by^{n+1}\]
with a positive time step size $\Delta t$, $\bW^n=\operatorname{diag}(y_i^n/y_i^{(2)})$ and $\bI$ denoting the identity matrix in $\R^{N\times N}$.
Hence, we find
$\by^{n+1}=\bm R^n \by^n$ with

\[\bm R^n=\biggl(\bI - \frac{\Delta t}{2}\bA(\bW^n +\bI)\biggr)^{-1}.\]
This shows that $\by^{n+1}$ depends nonlinearly on $\by^n$ even when a linear PDS is considered, which complicates the analysis significantly.

The system matrix $\bA=(a_{ij})\in\R^{N\times N}$ of a positive and conservative linear PDS written in the form \eqref{eq:lin_sys} must satisfy $a_{ii}\leq 0$, $a_{ij}\geq 0$ for $i\ne j$ and $\sum_{i=1}^N a_{ij}=0$ for $j=1,\dots,N$.
Hence, the system
\begin{align}\
\mbfy'=\bA\mbfy,\quad \bA=\begin{pmatrix*}[r]
-a &b\\ a &-b
\end{pmatrix*},\quad a,b\geq 0,\quad a+b> 0\label{PDS_test}
\end{align}
represents all positive and conservative linear PDS of size $2\times 2$, except the case with $a=b=0$, which needs no stability analysis. 
In order to see that \eqref{PDS_test} is indeed a PDS, we set
\begin{align}\label{pij_dij_Test_eq}  
p_{12}(\mbfy)&=d_{21}(\mbfy)=by_2,&
p_{21}(\mbfy)&=d_{12}(\mbfy)=ay_1,&
p_{ii}(\mbfy)&=d_{ii}(\mbfy)=0,\ i\in\{1,2\}
\end{align}
and obtain $y_1'=p_{12}(\by)-d_{12}(\by)$ and $y_2'=p_{21}(\by)-d_{21}(\by)$.
The eigenvalues of $\bA$ are $\lambda=-(a+b)<0$ and $0$. Given an initial value $\by^0=(y_1^0,y_2^0)^T$, the solution of the initial value problem associated with \eqref{PDS_test} is
\begin{equation}
\mbfy(t)=\frac{y_1^0+y_2^0}{a+b}\begin{pmatrix}
b\\a 
\end{pmatrix}+\frac{a y_1^0-by_2^0}{a+b}\Vec{1\\ -1}e^{\lambda t}=\frac{1}{a+b}\begin{pmatrix}b+a e^{\lambda t} & b-be^{\lambda t}\\a-ae^{\lambda t}&a+b e^{\lambda t}\end{pmatrix}\mbfy^0. \label{exact_sol}
\end{equation}
Since $\lambda<0$ we have $\mbfy(t)>0$ for $t\geq 0$ if $\mbfy^0>0$, which shows that the PDS \eqref{PDS_test} is positive. Moreover, summation in \eqref{exact_sol} shows $y_1(t)+y_2(t)=y_1^0+y_2^0$ for all $t\geq 0$, which confirms that the PDS is also conservative. 

The system \eqref{PDS_test} is also considered in \cite{IKM21}, where it is used to study a linearization of second order MPRK schemes. Section~\ref{sec:stab_mprk22} extends the results of \cite{IKM21} to the nonlinear case.

In the following, we introduce a framework to study the Lyapunov stability of positive and conservative time integrators, when applied to \eqref{PDS_test}.  
Within this framework we analyze the stability of
the second order MPRK22($\alpha$) and MPRK22ncs($\alpha$) schemes introduced in \cite{KM18}. If we want to refer to both schemes we use MPRK22 schemes as an abbreviation.
The MPRK22 schemes are given by
\begin{subequations}\label{eq:MPRK22b}
	\begin{align}
	&\begin{aligned}
	\mathllap{y_i^{(1)}} &= y_i^n,
	\end{aligned}\label{eq:y_1=y_n}\\ 
	&\begin{aligned}
	\mathllap{y_i^{(2)}} &= y_i^n + \alpha\Delta t\sum_{j=1}^N\Vastl((1-\gamma)p_{ij}(\mbfy^{(1)})+p_{ij}(\mbfy^{(1)})\frac{y_j^{(2)}}{y_j^{(1)}}\gamma-d_{ij}(\mbfy^{(1)})\frac{y_i^{(2)}}{y_i^{(1)}}\Vastr),
	\end{aligned}\label{eq:gamma}\\ 
	&\begin{multlined}[b][.7\columnwidth]
	\mathllap{y_i^{n+1}} = y_i^n + \Delta t\sum_{j=1}^N\Vastl( \Biggl(\biggl(1-\frac1{2\alpha}\biggr) p_{ij}(\mbfy^{(1)})+\frac1{2\alpha} p_{ij}(\mbfy^{(2)})\Biggr)\frac{y_j^{n+1}}{(y_j^{(2)})^{\frac{1}{\alpha}}(y_j^{(1)})^{1-\frac{1}{\alpha}}}\\
	- \Biggl(\biggl(1-\frac1{2\alpha}\biggr) d_{ij}(\mbfy^{(1)})+ \frac1{2\alpha} d_{ij}(\mbfy^{(2)})\Biggr)\frac{y_i^{n+1}}{(y_i^{(2)})^{\frac{1}{\alpha}}(y_i^{(1)})^{1-\frac{1}{\alpha}}}\Vastr),
	\end{multlined}\label{eq:y_n_plus_1}
	\end{align}
\end{subequations}
for $i=1,\dots,N$ with $\alpha\geq\frac12$. The methods with $\gamma=1$ in \eqref{eq:gamma} are called MPRK22$(\alpha)$ schemes. With $\gamma=~0$ they are named MPRK22ncs$(\alpha)$ schemes. For MPRK22($\alpha$) schemes $\by^{(2)}$ is also conservative, in the sense that $\sum_{i=1}^N y_i^{(2)}=~\sum_{i=1}^N y_i^0$. This is not the case for MPRK22ncs($\alpha$) methods, where \enquote{ncs} is an abbreviation for \enquote{non conservaitve stages}, and will result in inferior stability properties as shown in Section~\ref{sec:stab_mprk22}.

The outline of the paper is as follows. In Section \ref{sec:stab_dyn_sys} we summarize the center manifold theory and prove the main theorem concerning the Lyapunov stability of general positive and conservative time integration schemes. This theorem is used to analyze the local stability of the schemes \eqref{eq:MPRK22b} when applied to the positive and conservative linear PDS \eqref{PDS_test} in Section~\ref{sec:stab_mprk22}. We prove that the MPRK22($\alpha$) methods are unconditionally stable. Whereas, MPRK22ncs$(\alpha)$ is also unconditionally stable for $\alpha\geq 1$ and requires time step restrictions in the case  $\alpha<1$. Finally, we provide numerical experiments confirming the theoretical results in Section \ref{Num Tests}.

\section{Center manifold theory and stability of positive and conservative time integration schemes}\label{sec:stab_dyn_sys}
In this section we recall the definitions of stable and asymptotically stable steady state solutions of differential equations and the corresponding definitions for fixed points of iteration schemes. We also recap theorems that are helpful to identify the stability properties of a given fixed point.  
These show that for hyperbolic fixed points stability is solely determined by the eigenvalues of the Jacobian of the underlying map, which is not true for non-hyperbolic fixed points. The center manifold theory is an important tool to investigate the stability of non-hyperbolic fixed points.
Using this theory we present Theorem~\ref{Thm_MPRK_stabil}, which provides for the first time a criteria to assess the stability of general positive and conservative schemes applied to \eqref{PDS_test}.

In the following, $\norm{\ \cdot\  }$ denotes an arbitrary norm in $\R^N$ and $\bD\bbf$ denotes the Jacobian of a map $\bbf$.
\begin{defn}\label{Def Lyapunov Cont}
Let $\by^*$ be a steady state solution of a differential equation $\by'={\bbf}(\by)$, that is ${\bbf}(\by^*)=\bzero$.
\begin{enumerate}
\item\label{item:lyap_stab} The steady state solution $\by^*$ is called \emph{Lyapunov stable} if, for any $\epsilon>0$, there exists a $\delta=\delta(\epsilon)>0$ such that $\norm{\by(0)-\by^*}<\delta$ implies $\norm{\by(t)-\by^*}<\epsilon$ for all $t\geq 0$.
\item If in addition to a), there exists a constant $c>0$ such that $\Vert \by(0)-\by^*\Vert<c$ implies  $\Vert \by(t)-\by^*\Vert \to 0$ for $t\to \infty$,  we call the steady state solution $\by^*$ \emph{asymptotically stable.}
\item A steady state solution that is not stable is said to be \emph{unstable}.
\end{enumerate}
\end{defn}
The PDS \eqref{PDS_test} has infinitely many steady state solutions, since every $\by^*$ in the nullspace of $\bA$, that is $\by^*=\theta (b,a)^T$ with $\theta\in\R$, is a steady state of \eqref{PDS_test}.
In geometrical terms, all steady states lie on the line $a y_1-by_2=0$ in the $y_1$-$y_2$-coordinate system. With respect to the asymptotic behavior of \eqref{exact_sol}, we see
\begin{equation*}
\lim_{t\to\infty}\mbfy(t)=\frac{y_1^0+y_2^0}{a+b}\Vec{b\\a}+\lim_{t\to\infty}\frac{a y_1^0-by_2^0}{a+b}\Vec{1\\ -1}e^{\lambda t}=\frac{y_1^0+y_2^0}{a+b}\begin{pmatrix}
b\\a 
\end{pmatrix},
\end{equation*}
since $\lambda<0$.
Thereby, given an initial value $\by^0=(y_1^0,y_2^0)^T$, the solution monotonically approaches the steady state
\begin{equation*}
\by^*=\frac{y_1^0+y_2^0}{a+b}\Vec{b\\a}
\end{equation*}
along the line $y_1+y_2=y_1^0+y_2^0$ in the $y_1$-$y_2$-coordinate system. 
Hence, the steady state solutions 	 of \eqref{PDS_test} cannot be asymptotically stable, as there are infinitely many other steady state solutions in every neighborhood of a steady state. But it can be shown that they are stable in the sense of Defintion~\ref{Def Lyapunov Cont}~\ref{item:lyap_stab}, see \cite[Theorem 3.23]{MR1912409}.

When applied to a differential equation, a time integration scheme should preserve as many properties of the differential equation as possible. In particular, the fixed points of the iteration scheme, should be the steady state solutions of the differential equation with equal stability properties. 
\begin{defn}\label{Def_Lyapunov_Diskr}
Let $\by^*$ be a fixed point of an iteration scheme $\by^{n+1}=\bg(\by^n)$, that is $\by^*=\bg(\by^*)$. 
\begin{enumerate}
\item\label{def:stab} The fixed point $\by^*$ is called \emph{Lyapunov stable} if, for any $\epsilon>0$, there exists a $\delta=\delta(\epsilon)>0$ such that $\norm{\by^0-\by^*}<\delta$ implies $\norm{\by^n- \by^*}<\epsilon$ for all $n\geq 0$.
\item If in addition to a), there exists a constant $c>0$ such that $\Vert \by^0-\by^*\Vert<c$ implies $\Vert \by^n-\by^*\Vert \to 0$ for $n\to \infty$, the fixed point $\by^*$ is called \emph{asymptotically stable.}
\item A fixed point that is not stable is said to be \emph{unstable}.
\end{enumerate}
\end{defn}
In the following, we will also briefly speak of stability instead of Lyapunov stability.

Next, we summarize theorems which are helpful to investigate the stability of fixed points of iteration schemes.
 
\begin{thm}[{\cite[Theorem 1.3.7]{SH98}}]\label{Thm:_Asym_und_Instabil}
Let  $\by^{n+1}=\bg(\by^n)$ be an iteration scheme with fixed point $\by^*$. Then
\begin{enumerate}
\item $\by^*$ is asymptotically stable if $\abs\lambda <1$ for all eigenvalues $\lambda$ of $\bD\bg(\by^*)$. 
\item $\by^*$ is unstable if $\abs {\lambda}>1$ for one eigenvalue $\lambda$ of $\bD\bg(\by^*)$.
\end{enumerate}
\end{thm}
The above theorem does not give any information if the spectral radius of $\bD\bg(\by^*)$ is equal to $1$.
Hence, it is reasonable to introduce the following definition. 
\begin{defn}[{\cite[Definition 1.3.6]{SH98}}]\label{Def hyperbolic}
A fixed point $\by^*$ of an iteration scheme $\by^{n+1}=\bg(\by^n)$ is called \emph{hyperbolic} if $\abs\lambda\ne 1$ for all eigenvalues $\lambda$ of $\bD\bg(\by^*)$. If a fixed point is not hyperbolic, it is called \emph{non-hyperbolic}.
\end{defn}
A generalization of Theorem~\ref{Thm:_Asym_und_Instabil} is the Hartman-Grobman Theorem, which states that a nonlinear iteration scheme and its linearization share the same behavior near hyperbolic fixed points, see \cite[Theorem 1.6.2]{SH98} for the precise statement. One can show, that for non-hyperbolic fixed points nonlinear terms have to be taken into account in order to investigate the stability. Thereby, the theory of center manifolds is an important tool and will be explained in the following section. 

\subsection{Center Manifold Theory}\label{Sec:Center_Manifold}
To study the stability of a non-hyperbolic fixed point $\by^*$ of an iteration scheme with $\mathcal C^2$-map $\bg$, we make use of an affine linear transformation\footnote{See the proof of Theorem~\ref{Thm_MPRK_stabil} for the details of this transformation.} to obtain a $\mathcal{C}^2$-map $\bG\colon D\to \R^{N}$, with $D\subset \R^{N}$ being a neighborhood of the origin, which has the form
\begin{equation}
\bG(\bw_1,\bw_2)=\Vec{\bU\bw_1+\bu(\bw_1,\bw_2)\\\bV\bw_2 +\bv(\bw_1,\bw_2)},\label{Form_of_G}
\end{equation}
with $\bw_1\in \R^m$, $\bw_2\in \R^l$ and $m+l=N$. The square matrices $\bU\in\R^{m\times m}$ and $\bV\in\R^{l\times l}$ are such that $\abs\lambda = 1$ for all eigenvalues $\lambda$ of $\bU$ and $\abs\mu <1$ for all eigenvalues $\mu$ of $\bV$. The functions $\bu$ and $\bv$ are in $\mathcal{C}^2$ and $\bu,\bv$ as well as their first order derivatives vanish at the origin, that is
\begin{align*}
\bu(\bzero,\bzero)&=\bzero, & \bD\bu(\bzero,\bzero)&=\bzero, & 
\bv(\bzero,\bzero)&=\bzero, & \bD\bv(\bzero,\bzero)&=\bzero,
\end{align*} where $\bzero$ stands for the zero vector or matrix of appropriate size, respectively. 
In particular, the fixed point $\by^*$ of $\bg$ is mapped to $\bzero$, which is a fixed point of $\bG$ with equal stability properties as $\by^*$. 
Hence, it is sufficient to study the stability of the origin with respect to $\bG$, which is a simplification due to the existence of a center manifold.

\begin{thm}[Center Manifold Theorem]\label{Thm:Ex_CM}
Let $\bG$ be defined as in \eqref{Form_of_G} and 
\begin{align}
\Vec{\bw_1^{n+1}\\\bw_2^{n+1}}=\bG(\bw_1^n,\bw_2^n)=\Vec{\bU\bw_1^n+\bu(\bw_1^n,\bw_2^n)\\\bV\bw_2^n +\bv(\bw_1^n,\bw_2^n)},\quad \Vec{\bw_1^0\\ \bw_2^0}\in \R^{m+l}\label{Form_of_G_transformed}.
\end{align}
\begin{enumerate}
\item\label{item:existence}  (Existence):
There exists a center manifold for $\bG$, which is locally representable as the graph of a function $\bh\colon\R^m\to \R^l$.  This means, for some $\epsilon>0$ there exists a $\mathcal{C}^1$-function $\bh\colon\R^m\to\R^l$ with $\bh(\bzero)=\bzero$ and $\bD\bh(\bzero)=\bzero$ such that $\Vert \bw_1^0\Vert <\epsilon$ and $(\bw_1^1, \bw_2^1)^T=\bG(\bw_1^0,\bh(\bw_1^0))$ implies $\bw_2^1=\bh(\bw_1^1)$.
\item\label{item:attractivity}  (Local Attractivity): Let $(\bw_1^n,\bw_2^n)^T$, $n\in\N_0$ represent the sequence generated by \eqref{Form_of_G_transformed}.  If $\Vert \bw_1^n\Vert,\Vert \bw_2^n\Vert<\epsilon$ for all $n\in \N_0$, then the distance of $(\bw_1^n,\bw_2^n)$ to the center manifold tends to zero for $n\to \infty$, i.\,e.\ $\Vert \bw_2^n-\bh(\bw_1^n)\Vert\to 0$ for $n\to \infty$.
\end{enumerate}
\end{thm}
\begin{proof}
See \cite[Theorem 2.1]{mccracken1976hopf}, \cite[Theorem 6]{carr1982}, \cite[Theorem 4]{Osipenko2009} for existence and \cite[Theorem 2.1]{mccracken1976hopf}, \cite[Chapter V, Theorem 2]{iooss1979} for local attractivity.\qedhere
\end{proof}
The existence of a center manifold allows the study of a system with reduced dimensionality to determine the stability properties of the origin. Restricting the iteration \eqref{Form_of_G_transformed} to the center manifold, i.\,e. $\bw_2^0=\bh(\bw_1^0)$, gives
\begin{equation}\label{Flow_center_manifold}\bw_1^{n+1}= \mathcal G(\bw_1^{n}),\quad \mathcal G(\bw_1)=\bU \bw_1 + \bu(\bw_1,\bh(\bw_1))
\end{equation}
for $\norm{\bw_1^n}<\epsilon$.
The next theorem states that stability of the origin with respect to $\mathcal G$ already implies stability of the origin with respect to $\bG$.
\begin{thm}\label{Thm:Stab_CM}
\cite[Theorem 8]{carr1982} (Stability): Suppose the fixed point $\bzero\in \R^m$ of $\mathcal G$ from \eqref{Flow_center_manifold} is stable, asymptotically stable or unstable. Then the fixed point $\bzero \in \R^N$ of $\bG$ from \eqref{Form_of_G_transformed} is stable, asymptotically stable or unstable.
\end{thm}
In summary,  the stability of a non-hyperbolic fixed point $\by^*\in\R^N$ of a map $\bg$ can be determined by investigating the fixed point $\bzero\in\R^m$ of $\mathcal G$, which has a lower complexity due to the reduced dimension $m<N$.

To actually calculate the center manifold we need to solve
\begin{equation*}
(\bw_1^1, \bh(\bw_1^1))^T=\bG(\bw_1^0,\bh(\bw_1^0))=\Vec{\bU\bw_1^0+\bu(\bw_1^0,\bh(\bw_1^0))\\\bV\bh(\bw_1^0) +\bv(\bw_1^0,\bh(\bw_1^0))},
\end{equation*} 
which can be rewritten as
\[\bh(\bU\bw_1^0+\bu(\bw_1^0,\bh(\bw_1^0)))=\bV\bh(\bw_1^0) +\bv(\bw_1^0,\bh(\bw_1^0))\]
or 
\[\bh(\bG(\bw_1^0,\bh(\bw_1^0))_1)=\bG(\bw_1^0,\bh(\bw_1^0))_2.\]
The above invariance property offers a way to approximate the center manifold up to an arbitrary order.
\begin{thm}\label{Thm:Comp_func_h}
\cite[Theorem 7]{carr1982}
Let $\bh$ be a center manifold for $\bG$ and $\bm\Phi\in \mathcal{C}^1(\R^m, \R^l)$ with $\bm\Phi(\bzero)=\bzero$ and $\bD\bm\Phi(\bzero)=\bzero$. If
\[\bm\Phi(\bG(\bw_1,\bm\Phi(\bw_1))_1)-\bG(\bw_1,\bm\Phi(\bw_1))_2=\mathcal{O}(\lvert \bw_1 \rvert^q)\]  as $\bw_1\to \bzero$ for some $q>1$, then $\bh(\bw_1)=\bm\Phi(\bw_1)+\mathcal{O}(\lvert \bw_1 \rvert^q)$ as $\bw_1\to \bzero$.
\end{thm}

Now, we consider a general positive and conservative iteration scheme $\by^{n+1}=\bg(\by^n)$, i.\,e. $\by^n>\bzero$ for all $n\in\N$ if $\by^0>\bzero$ and $\norm{\by^{n+1}}_1=\norm{\by^n}_1$, in two dimensions.
We further assume that all fixed points of the iteration scheme are located on a line through the origin, which is the case for the steady states of \eqref{PDS_test}.  
Under these circumstances, Theorem~\ref{Thm_MPRK_stabil} gives a sufficient condition for stability of the iteration scheme based on the eigenvalues of the corresponding Jacobian $\bD\bg(\by^*)$. Furthermore, the theorem states that stability implies convergence towards a fixed point with equal 1-norm.

The following lemma is used in the proof of Theorem~\ref{Thm_MPRK_stabil}.
\begin{lem}\label{Lem:alpha1=1}
Let $\by^0=(y_1^0,y_2^0)^T>\bzero$ and $L=\{\by\in\R^2_{>0}\mid\norm{\by}_1=\norm{\by^0}_1\}$,
then
\[L=\{\by\in\R^2\mid \by= \by^0 + s\bby,\ -y_1^0<s<y_2^0\},\]
where $\bby=(1,-1)^T$. If in addition, $\bg\from\R^2_{>0}\to\R^2_{>0}$ is a conservative map, i.\,e. $\norm{\bg(\by)}_1=\norm{\by}_1$, then for every $\by^0+t\bby>\bzero$ with $t\in\R$, there exists a $s(t)$ with $-y_1^0<s(t)<y_2^0$, such that
\[\bg(\by^0+t\bby)=\by^0+s(t)\bby.\]
\end{lem}
\begin{proof}
Let $\by^0,\by>\bzero$, then the condition $\norm{\by}_1=\norm{\by^0}_1$ is equivalent to $y_1+y_2=y_1^0+y_2^0$. Geometrically, this describes a line in the $y_1$-$y_2$-coordinate system with normal vector $(1,1)^T$. A parameter form of this line is given by $\by=\by^0 + s(t)\bby$, $s(t)\in\R$. To ensure $\by>\bzero$, we must have $-y_1^0<s(t)<y_2^0$.

Now, define $\bw=\bg(\by^0+t\bby)$ with $\by^0+t\bby>\bzero$. Since $\norm{\bw}_1=\norm{\by^0+t\bby}_1=\norm{\by^0}_1$, we have $\bg(\by^0+t\bby)=\bw=\by^0+s(t)\bby$ with $-y_1^0<s(t)<y_2^0$.
\end{proof}
Next, we present the main theorem of this section, which provides criteria to assess the stability of general positive and conservative schemes.
Application of a general positive and conservative scheme to \eqref{PDS_test} results in a nonlinear iteration $\by^{n+1}=\bg(\by^n)$, for
which the steady states $\by^*$ of \eqref{PDS_test} should be non-hyperbolic fixed points of $\bg$.
The theorem shows that even in this nonlinear case the investigation of the eigenvalues of the Jacobian $\bD\bg(\by^*)$ is sufficient to analyze stability. 
To the authors knowledge there are no similar results focusing on general positive and conservative schemes in the literature even though the statements are of fundamental importance 
\begin{thm}\label{Thm_MPRK_stabil}
Let $\bg\in \mathcal{C}^2(\R^2_{>0})$ with fixed point $\by^*>\bzero$, such that all $r\by^*$ are fixed points of $\bg$ for all $r>0$. In addition, let the iterates of the iteration scheme  $\by^{n+1}=\bg(\by^n)$ satisfy $\lVert \by^{n+1} \rVert_1=\lVert \by^n \rVert_1$ for all $n\in \N_0$. Then, the spectrum of the Jacobian $\bD\bg(\by^*)$ is $\sigma(\bD\bg(\by^*))=\{1,R\}$ with $R\in\R$, and the following statements apply.
\begin{enumerate}
\item If $\lvert R \rvert <1$, then $\by^*$ is stable.
\item If $\by^*$ is stable, then there exists a $\delta>0$, such that $\Vert \by^{0} \Vert_1=\Vert \by^* \Vert_1$ and $\norm{\by^0-\by^*}<\delta$ imply $\by^n\to \by^*$ as $n\to \infty$.
\end{enumerate}
\end{thm}
\begin{proof}
Throughout this proof, we use the notation $\bby=(1,-1)^T$ and $\be_1=(1,0)^T$, $\be_2=(0,1)^T$ to denote the standard unit vectors.

First, we compute the eigenvalues and eigenvectors of $\bD\bg(\by^*)$. Since $\bg$ is differentiable in $\by^*$ the directional derivatives $\partial_\bv \bg(\by^*)=\bD\bg(\by^*)\bv$ exist for all directions $\bv\in\R^2$. Hence,
\[\bD\bg(\by^*)\by^*=\partial_{\by^*}\bg(\by^*)=\lim_{h\to 0}\frac1h\bigl(\bg(\by^*+h\by^*)-\bg(\by^*)\bigr)=\lim_{h\to 0}\frac1h\bigl(\bg((1+h)\by^*)-\by^*\bigr). \]
As $r\by^*>\bzero$ is a fixed point of $\bg$ for all $r>0$, we see
\[\bD\bg(\by^*)\by^*=\lim_{h\to 0}\frac1h\bigl((1+h)\by^*-\by^*\bigr)=\by^*. \]
Thus, $\by^*$ is an eigenvector of $\bD\bg(\by^*)$ with associated eigenvalue $1$.
To compute the other eigenvalue and eigenvector, we consider the directional derivative
\begin{equation}\label{eq:directional_derivative}
\bD\bg(\by^*)\bby=\partial_{\bby}\bg(\by^*)=\lim_{h\to 0}\frac1h\bigl(\bg(\by^*+h\bby)-\bg(\by^*)\bigr). \end{equation}
Since $\norm{\bg(\by)}_1=\norm{\by}_1$, we can use Lemma~\ref{Lem:alpha1=1} to see
\[\bg(\by^*+h\bby)=\by^*+ s(h)\bby,\]
for sufficiently small $h$ and some function $s\from\R\to\R$, $h\mapsto s(h)$.
Inserting this into \eqref{eq:directional_derivative} yields
\begin{equation*}
\bD\bg(\by^*)\bby=\lim_{h\to 0}\frac1h\bigl(\by^* + s(h)\bby-\by^*\bigr)=\biggl(\lim_{h\to 0}\frac{s(h)}h\biggr)\bby. 
\end{equation*}
The above limit exists, since $\bg$ is differentiable in $\by^*$. Setting $R=\lim_{h\to 0}\frac{s(h)}h\in\R$, we see $\bD\bg(\by^*)\bby=R\bby$, i.\,e. $\bby$ is an eigenvalue of $\bD\bg(\by^*)$ with corresponding eigenvalue $R$.
Hence, the spectrum of $\bD\bg(\by^*)$ is given by $\sigma(\bD\bg(\by^*))=\{1,R\}$.
Introducing the matrix of eigenvectors
\begin{equation}\label{eq:Smatrix}
\bS=(\by^* \bby),
\end{equation}
which is invertible, since $\bby$ cannot be a multiple of the positive vector $\by^*$, we obtain
\begin{equation}\label{eq:DG_digaonal}
\bS^{-1}\bD\bg(\by^*)\bS=\diag(1,R),
\end{equation}
where $\diag(\by)\in\R^{2\times 2}$ denotes the diagonal matrix with $(\diag(\by))_{ii}=y_i$ for $i=1,2$.
\begin{enumerate} 
\item\label{item:stability}
In this part we assume $\abs{R}<1$ and use the center manifold Theorem~\ref{Thm:Ex_CM}~\ref{item:existence} in combination with Theorem~\ref{Thm:Stab_CM} to conclude that this implies  that $\by^*$ is a stable fixed point.
The theorem requires a map $\bG$ of form \eqref{Form_of_G}, which shall be obtained from $\bg$ by means of an affine linear transformation.

We consider the affine transformation $\bT\from\R^2\to\R^2$, $\by\mapsto\bw=\bT(\by)=\bS^{-1}(\by-\by^*)$, where $\bS$ is given in \eqref{eq:Smatrix} and
the inverse transformation $\bT^{-1}$ is given by $\bT^{-1}(\bw)=\bS\bw+\by^*$.
By construction, the line segment of fixed points $\{\by\in\R_{>0}^2\mid\by=r\by^*,\ r>0\}$ is mapped onto the $w_1$-axes, as
$\bT(r\by^*)=\bS^{-1}(r\by^*-\by^*)=(r-1)\bS^{-1}\by^*=(r-1)\be_1$.
The line segment $\{\by\in\R^2_{>0}\mid\norm{\by}_1=\norm{\by^*}_1\}=\{\by\in\R^2\mid \by= \by^* + s\bby,\ -y_1^*<s<y_2^*\}$, see Lemma~\ref{Lem:alpha1=1}, is mapped onto the $w_2$-axes, since
$\bT(\by^*+s\bby)=\bS^{-1}(\by^*+s\bby-\by^*)=s\bS^{-1}\bby=s\be_2$.
In particular, $\by^*$ is mapped to the origin and parallel lines are mapped onto parallel lines, since $\bT$ is affine.
See Figure~\ref{fig:affine_transf} for a sketch of this situation.
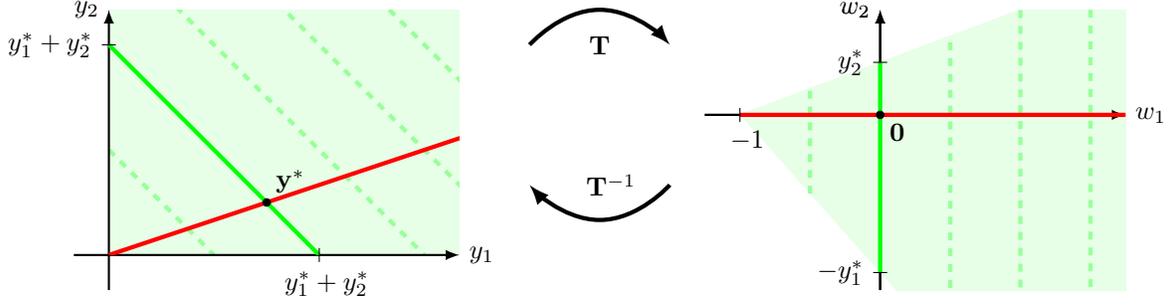
\begin{figure}
\begin{center}
\begin{scaletikzpicturetowidth}{\textwidth}
\begin{tikzpicture}[scale=\tikzscale]
\kosgrid{-.5}{5}{-.5}{3.5}{$y_1$}{$y_2$}{0,0}
\begin{scope}
\clip (0,0) rectangle (5,3.5);
\filldraw[green!10] (0,0) -- (5,0) -- (5,3.5) -- (0,3.5);
\draw[green!40,lineplot,ultra thick, dashed] (0,1.5) -- (1.5,0);
\draw[green!40,lineplot,ultra thick, dashed] (0,4.5) -- (4.5,0);
\draw[green!40,lineplot,ultra thick, dashed] (0,6) -- (6,0);
\draw[green!40,lineplot,ultra thick, dashed] (0,7.5) -- (7.5,0);
\kosgrid{-.5}{5}{-.5}{3.5}{$y_1$}{$y_2$}{0,0}
\draw[red,lineplot,ultra thick] (0,0) -- (6,2);
\draw[green,lineplot,ultra thick] (0,3) -- (3,0);
\filldraw[black] (2.25,0.75) circle[radius=1.5pt] node[above right]{$\by^*$};
\end{scope}
\xticks{3}{\tickl}{$y_1^*+y_2^*$}{below, xshift=+.1cm}
\yticks{3}{\tickl}{$y_1^*+y_2^*$}{left}
\begin{scope}[shift={(11,2)},xscale=2]
\kosgrid{-1.25}{1.75}{-2.5}{1.5}{$w_1$}{$w_2$}{0,0}
\begin{scope} 
\clip (-1.25,-2.5) rectangle (1.75,1.5);
\filldraw[green!10] (-1,0) -- (2,2.25) -- (2,-6.75);
\kosgrid{-1.25}{1.75}{-2.5}{1.5}{$w_1$}{$w_2$}{0,0}
\draw[green!40,lineplot,ultra thick,dashed] (-.5,-1.125) -- (-.5,.375);
\draw[green!40,lineplot,ultra thick,dashed] (0.5,-3.3750) -- (.5,1.125);
\draw[green!40,lineplot,ultra thick,dashed] (1,-4.5) -- (1,1.5);
\draw[green!40,lineplot,ultra thick,dashed] (1.5,-5.6250) -- (1.5,1.8750);
\draw[green,lineplot,ultra thick] (0,-2.25) -- (0,.75);
\draw[red,lineplot,ultra thick] (-1,0) -- (1.75,0);
\filldraw[black] (0,0) circle[radius=1.5pt,xscale=0.5] node[below right]{$\bzero$};
\end{scope}
\xticks{-1}{\tickl}{$-1$}{below, xshift=+.1cm}
\yticks{-2.25}{\tickl/2}{$-y_1^*$}{left}
\yticks{0.75}{\tickl/2}{$y_2^*$}{left}
\end{scope}
\draw [thick arrow] plot [smooth, tension=1.0] coordinates { (6,3)  (7,3.5) (8,3) };
\node at (7,3) {$\bT$};
\draw [thick arrow] plot [smooth, tension=1.0] coordinates { (8,1)  (7,0.5) (6,1) };
\node[xshift=.15cm] at (7,1) {$\bT^{-1}$};
\end{tikzpicture}
\end{scaletikzpicturetowidth}
\end{center}
\caption{Visualisation of the affine transformation $\bT$. The red and green line segments are mapped onto each other.}
\label{fig:affine_transf}
\end{figure}

Now we define
\begin{equation}
\bG\from \bT(\R_{>0}^2)\to \bT(\R_{>0}^2),\quad \bG(\bw) = \bT(\bg(\bT^{-1}(\bw)))\label{eq:TgTinv}
\end{equation}
and observe that the origin is a fixed point of $\bG$. To represent $\bG$ in the form \eqref{Form_of_G}, we write $\bg$ as
\begin{equation}
\bg(\by)=\bg(\by^*)+\bD\bg(\by^*)(\by-\by^*)+\bR(\by)=\by^*+\bD\bg(\by^*)(\by-\by^*)+\bR(\by),\label{yn+1_Lagrange_remainder}
\end{equation}
where the remainder $\bR(\by)$ can be written in Lagrange form
\begin{align}
(\bR(\by))_i=\frac{1}{2}(\by-\by^*)^T\bH g_i(\by^*+c(\by-\by^*))(\by-\by^*),\quad i=1,2,\label{Lagrange_remainder}
\end{align}
with $\bH g_i$ denoting the Hessian of $g_i$ and $c\in(0,1)$ depends on $\by$ and $\by^*$. 
This is possible because $\bg$ is assumed to be in $\mathcal C^2$ on the convex set $\R^2_{>0}$.
 In particular, we have
\begin{align}
{\bR}(\by^*)=\bzero,\quad
\bD {\bR}(\by^*)=\bzero.\label{R(0)=0,DR(0)=0}
\end{align}
By inserting \eqref{yn+1_Lagrange_remainder} in \eqref{eq:TgTinv} we obtain
\begin{equation*}
\bG(\bw) = \bS^{-1}\bigl(\bD\bg(\by^*)(\bT^{-1}(\bw)-\by^*)+\bR(\bT^{-1}(\bw))\bigr)=\bS^{-1}\bD\bg(\by^*)\bS\bw + \bS^{-1}\bR(\bT^{-1}(\bw))
\end{equation*}
and using \eqref{eq:DG_digaonal}, we see
\begin{equation}\label{eq:G}
\bG(\bw) = \diag(1,R)\bw + \bS^{-1}\bR(\bT^{-1}(\bw)).
\end{equation}
With $\bw=(w_1,w_2)^T$ this can be rewritten as
\begin{equation}\label{eq:G_form}
\bG(w_1,w_2)=\Vec{U w_1+u(w_1,w_2)\\V w_2+v(w_1,w_2)}
\end{equation} 
where 
\begin{align}\label{eq:UVuv}
U&=1, & V&= R, & u(w_1,w_2)&=\bigl(\bS^{-1}{\bR}(\bT^{-1}(w_1,w_2))\bigr)_1, & v(w_1,w_2)&=\bigl(\bS^{-1}{\bR}(\bT^{-1}(w_1,w_2))\bigr)_2.
\end{align}
The eigenvalues of $U$ have absolute value 1 and those of $V$ have absolute value $\abs{R}<1$. Furthermore, we conclude from \eqref{R(0)=0,DR(0)=0} that
$u(0,0) = v(0,0)=0$, since $\bR(\by^*)=\bzero$, and $\bD u(0,0) = \bD v(0,0)=0$, since $\bD\bR(\by^*)=\bzero$.
Thus, \eqref{eq:G} is of form \eqref{Form_of_G}.

Now, the center manifold theorem~\ref{Thm:Ex_CM}~\ref{item:existence} states that for some $\epsilon>0$ there exists a $\mathcal{C}^2$ function $h\from\R\to \R$ with $h(0)=0$ and $h'(0)=0$ , such that $(w_1^1,w_2^1)^T=\bG(w_1^0,h(w_1^0))$ and  $\abs{w_1^0} <\epsilon$ implies $w_2^1=h(w_1^1)$. Furthermore, $\epsilon$ can be further reduced, if necessary, to guarantee $(w_1,h(w_1))\in \bT(\R_{>0}^2)$ for $\lvert w_1\rvert <\epsilon$ as $h$ is continuous with $h(0)=0$.
We know from Lemma~\ref{Lem:alpha1=1}, that for all $\by^0>\bzero$ the map $\bg$ is invariant on the line segment $\{\by\in\R^2_{>0}\mid\norm{\by}_1=\norm{\by^0}_1\}=\{\by\in\R^2\mid \by= \by^0 + s\bby,\ -y^0_1<s<y^0_2\}$. The transformation $\bT$ maps such a line segment onto a segment of a line parallel to the $w_2$-axes, on which $\bG$ must then be invariant. To verify this, we use Lemma~\ref{Lem:alpha1=1} and compute
\begin{align*}
\bG(w_1,w_2)&=\bG(w_1\be_1 + w_2\be_2) =  \bT(\bg(\bT^{-1}(w_1\be_1+w_2\be_2)))
=  \bT(\bg(w_1\bS\be_1+w_2\bS\be_2+\by^*))\\
&=  \bT(\bg((1+w_1)\by^*+w_2\bby))
=  \bT((1+w_1)\by^*+ s(w_2) \bby)\\
& = \bS^{-1}(w_1\by^*+s(w_2)\bby)=w_1\be_1+s(w_2)\be_2=(w_1,s(w_2))^T
\end{align*}
for some suitable $s=s(w_2)\in\R$. 
In particular, we have shown
\begin{equation}\label{eq:G1_constant}
\bG(w_1,w_2)_1 = w_1.
\end{equation}

We can now consider the iteration scheme 
\begin{equation*}
w_1^{n+1}= \mathcal G(w_1^n),\quad \mathcal G(w_1)=U w_1 + u(w_1,h(w_1))
\end{equation*}
for $\abs{w_1^n}<\epsilon$, where $U$ and $u$ are given in \eqref{eq:UVuv}.
According to Theorem~\ref{Thm:Stab_CM}, the fixed point $\bzero\in\R^2$ of $\bG$ is stable, if the fixed point $0\in\R$ is a stable fixed point of $\mathcal G$.
From \eqref{eq:G_form} and \eqref{eq:G1_constant} we see
\begin{equation*}
\mathcal G(w_1)=\bG(w_1,h(w_1))_1=w_1,
\end{equation*}
which implies $w_1^{n}=\mathcal G(w_1^{n-1})=w_1^0$ for all $n\in\N$ and every $w_1^0$ with $\abs{w_1^0}<\epsilon$. Consequently, for every $\widetilde\epsilon>0$ we define $\widetilde\delta = \min(\widetilde\epsilon,\epsilon)$ to obtain that $\abs{w_1^0}<\widetilde\delta$ implies $\abs{w_1^n}=\abs{w_1^0}<\widetilde\delta\leq \widetilde\epsilon$. Thus, $0$ is a stable fixed point of $\mathcal G$ in the sense of Definition~\ref{Def_Lyapunov_Diskr}~\ref{def:stab}.
Furthermore, by Theorem~\ref{Thm:Stab_CM} the fixed point $\bzero\in\R^2$ of $\bG$ is stable as well.

As a last step, we show that the above implies that $\by^*$ is a stable fixed point of $\bg$.
We know that $\bzero$ is a stable fixed point of the iteration scheme $\bw^{n+1}=\bG(\bw^n)$, that is for every $\epsilon_w>0$ exists $\delta_w>0$ such that $\norm{\bw^0}<\delta_w$ implies $\norm{\bw^n}<\epsilon_w$. Now, let $\epsilon>0$, $\epsilon_w=\epsilon/\norm{\bS}$ and  $\delta=\delta_w/\norm{\bS^{-1}}$. If $\norm{\by^0-\by^*}<\delta$, then $\norm{\bw^0}=\norm{\bT(\by^0)}=\norm{\bS^{-1}(\by^0-\by^*)}\leq \norm{\bS^{-1}}\norm{\by^0-\by^*}<\norm{\bS^{-1}}\delta=\delta_w$ and consequently $\norm{\bw^n}<\epsilon_w$. Furthermore, $\bw^n=\bT(\by^n)=\bS^{-1}(\by^n-\by^*)$ is equivalent to $\bS\bw^n=\by^n-\by^*$ and hence,  $\norm{\by^n-\by^*}\leq\norm{\bS}\norm{\bw^n}<\norm{\bS}\epsilon_w=\epsilon$. Thus, we have shown that $\by^*$ is a stable fixed point of the iteration $\by^{n+1}=\bg(\by^n)$.

\item
In the following we use that the center manifold is given by $\{(w_1,w_2)\in\R^2\mid w_2=0,\ \lvert w_1\rvert <\epsilon\}$, i.\,e. $h(w_1)=0$, for a sufficiently small $\epsilon>0$.  This can be shown with Theorem~\ref{Thm:Comp_func_h} as follows.
The function $\Phi\from\R\to\R$, $\Phi(w_1)=0$ satisfies $\Phi(0)=\Phi'(0)=0$. 
Furthermore, all points on the $w_1$-axes are fixed points of $\bG$, since
\begin{align*}
\bG(w_1,0)&=\bG(w_1\be_1) =  \bT(\bg(\bT^{-1}(w_1\be_1)))
=  \bT(\bg(w_1\bS\be_1+\by^*))\\
&=  \bT(\bg((1+w_1)\by^*))
=  \bT((1+w_1)\by^*)\\
& = \bS^{-1}(w_1\by^*)=w_1\be_1=(w_1,0)^T.
\end{align*}
Hence, it follows that
\begin{equation*}
\Phi(\bG(w_1,\Phi(w_1))_1) - \bG(w_1,\Phi(w_1))_2=-\bG(w_1,0)_2=0.
\end{equation*}
By Theorem~\ref{Thm:Comp_func_h} $\Phi$ is an approximation of $h$ for any order $q>1$. Thus, $h(w_1)=\Phi(w_1)=0$ for $\abs{w_1}<\epsilon$.

Now, we prove that the iteration scheme $\by^{n+1}=\bg(\by^n)$ is locally convergent to $\by^*$, if the starting value $\by^0$ is sufficiently close to $\by^*$ and satisfies $\norm{\by^0}_1=\norm{\by^*}_1$. Taking advantage of the transformation $\bT$, this is equivalent to proving that $\bw^{n+1}=\bG(\bw^n)$ is locally convergent to $\bzero$, if the starting value $\bw^0$ is sufficiently close to $\bzero$ and satisfies $w_1^0=0$. 
Moreover, since $\bG$ leaves the first component of its argument fixed, see \eqref{eq:G1_constant}, we only need to show $w_2^n\to 0$ for $n\to\infty$.
According to Theorem~\ref{Thm:Ex_CM}~\ref{item:attractivity}
 the distance of $(w_1^n,w_2^n)\in\R^2$ to the center manifold tends to zero for $n\to \infty$, i.\,e.\ $\lvert w_2^n\rvert\to 0$ for $n\to \infty$, if $\lvert w_1^n\rvert,\lvert w_2^n\rvert<\epsilon$ for all $n\in \N_0$ and some sufficiently small $\epsilon>0$. Finally, since the origin is a stable fixed point of $\bG$, as shown in \ref{item:stability}, there exists $\delta>0$ such that $\norm{\bw^0}_\infty<\delta$ implies $\norm{\bw^n}_\infty=\max\{\abs{w_1^n},\abs{w_2^n}\}<\epsilon$ for all $n \in\N$. If necessary, $\norm{\bw^0}_\infty <\epsilon$ can be assured by choosing $\delta<\epsilon$.   Altogether, this proves $\norm{\bw^n}_\infty=\max\{\abs{w_1^n},\abs{w_2^n}\}<\epsilon$ for all $n \in\N_0$ and thus the assertion.\qedhere
\end{enumerate}
\end{proof}
We like to highlight that Theorem~\ref{Thm_MPRK_stabil} is valid for general positive and conservative schemes and not restricted to MPRK schemes. 

\section{Stability of MPRK22 schemes}\label{sec:stab_mprk22}
In this section we use Theorem~\ref{Thm_MPRK_stabil} to examine the stability properties of MPRK22($\alpha$) and MPRK22ncs($\alpha$) schemes. The main task is to express the schemes in the form $\by^{n+1}=\bg(\by^n)$ and to compute the eigenvalues of the Jacobian $\bD\bg(\by^*)$ for a given fixed point $\by^*$.
We show that MPRK22($\alpha$) schemes for all permissible $\alpha$ and MPRK22ncs($\alpha$) schemes with $\alpha\geq 1$  are unconditionally stable. Hence, the iterates of these methods locally converge towards a fixed point with equal 1-norm.
Furthermore, for MPRK22ncs($\alpha$) schemes with $\alpha< 1$ time step restrictions are necessary to achieve the same behavior.

To give a clear representation of the MPRK iterations in the form $\by^{n+1}=\bg(\by^n)$, we define some auxiliary matrices and functions. First, we split the matrix $\bA$ into a production part $\bA_P$ and a destruction part $\bA_D$, i.\,e.
\begin{equation}\label{eq:A_splitting}\bA=\bA_P-\bA_D,\quad \bA_P = \bA-\diag(\bA)= \Mat{0 & b\\a & 0},\quad \bA_D=-\diag(\bA)=\Mat{a&0\\0&b}.
\end{equation}
Based on this splitting we define the matrices 
\begin{equation} 
\bB_\gamma =\gamma\bigl(\bI-\alpha\Delta t\bA\bigr)^{-1}+ (1-\gamma)\bigl(\bI+\alpha\Delta t\bA_D\bigr)^{-1}\bigl(\bI+\alpha\Delta t\bA_P\bigr)\label{def_B_gamma}
\end{equation}
and
\begin{equation}
\bC_\gamma=\biggl(1-\frac{1}{2\alpha}\biggr)\bI+\biggl(\frac{1}{2\alpha}\biggr)\bB_\gamma, \label{def_C}
\end{equation}
where $\bI\in\R^{2\times 2}$ denotes the identity matrix.
Furthermore, we need the nonlinear functions 
$\bssigma(\by)=(\sigma_1(\by),\sigma_2(\by))^T$ and $\btau(\by)=(\tau_1(\by),\tau_2(\by))^T$ with
\begin{equation}\label{def_sigma}
\sigma_i(\by)=(\bB_\gamma \by)_i^{\frac{1}{\alpha}}(y_i)^{1-\frac{1}{\alpha}}
\end{equation}
and
\begin{equation}\label{def_h}
\tau_i(\by)= \frac{(\bC_\gamma\by)_i}{\sigma_i(\by)}
\end{equation} 
for $i=1,2$.
\begin{prop}\label{prob:f}
The MPRK22 schemes \eqref{eq:MPRK22b} applied to the test equation \eqref{PDS_test}, can be written in the form $\by^{n+1}=\bg(\by^n)$, where
the map $\bg\colon\R_{>0}^2\to \R_{>0}^2$ is given by
\begin{equation}
\bg(\by)=\biggl(\bI + \frac{\Delta t}{1 + \Delta t(a\tau_1(\by)+b\tau_2(\by))}\bA\diag\bigl(\btau(\by)\bigr)\biggr)\by,\label{def_g(y)}
\end{equation}
with $\btau$ defined by \eqref{def_h}.
\end{prop}
\begin{proof}
We start with \eqref{eq:gamma} to express $\by^{(2)}$ in terms of $\by^n$. 
Using \eqref{pij_dij_Test_eq} and \eqref{eq:y_1=y_n}, the stage value \eqref{eq:gamma}, when  applied to \eqref{PDS_test}, is given by
\begin{equation}\label{eq:y^2_matrix_vector}
\begin{split}
y_1^{(2)}&=y_1^n + \alpha\Delta t\Bigl((1-\gamma)b y_2^n + \gamma b y_2^{(2)}-ay_1^{(2)}\Bigr),\\
y_2^{(2)}&=y_2^n + \alpha\Delta t\Bigl((1-\gamma)a y_1^n + \gamma ay_1^{(2)}-by_2^{(2)}\Bigr).
\end{split}
\end{equation}
Incorporating the splitting \eqref{eq:A_splitting}, the linear system \eqref{eq:y^2_matrix_vector} becomes
\begin{equation*}
\by^{(2)} = \by^n + \alpha\Delta t\Bigl((1-\gamma)\bA_P\by^n+\gamma \bA_P\by^{(2)}-\bA_D\by^{(2)}\Bigr)
\end{equation*}
 and solving for $\by^{(2)}$ shows

\begin{equation*}
\by^{(2)}=\Bigl(\bI-\alpha\Delta t(\gamma \bA_P-\bA_D)\Bigr)^{-1}\Bigl(\bI+\alpha\Delta t(1-\gamma)\bA_P\Bigr)\by^n.
\end{equation*}
In the case $\gamma=1$, we have
\begin{equation}\label{eq:y^2_delta_1}
\by^{(2)}=\bigl(\bI-\alpha\Delta t\bA\bigr)^{-1}\by^n,
\end{equation}
and for $\gamma = 0$, we get
\begin{equation}\label{eq:y^2_delta_0}
\by^{(2)}=\bigl(\bI+\alpha\Delta t\bA_D\bigr)^{-1}\bigl(\bI+\alpha\Delta t\bA_P\bigr)\by^n.
\end{equation}
Using \eqref{def_B_gamma},
the cases \eqref{eq:y^2_delta_1} and \eqref{eq:y^2_delta_0} can be combined to obtain
\begin{equation}\label{eq:y2_Bdelta_yn}
\by^{(2)}=\bB_\gamma\by^n.
\end{equation}
Next, we consider \eqref{eq:y_n_plus_1} to express $\by^{n+1}$ in terms of $\by^n$. 
With \eqref{pij_dij_Test_eq} and \eqref{def_sigma} the approximation step \eqref{eq:y_n_plus_1} reads
\begin{equation}\label{eq:y_n_plus_1_test}
\begin{split}
y_1^{n+1} &= y_1^n + \Delta t\vastl(\Biggl(\biggl(1-\frac1{2\alpha}\biggr) b y_2^n+\biggl(\frac1{2\alpha}\biggr)by_2^{(2)}\Biggr)\frac{y_2^{n+1}}{\sigma_2(\by^n)} - \Biggl(\biggl(1-\frac1{2\alpha}\biggr) a y_1^n+\biggl(\frac1{2\alpha}\biggr)ay_1^{(2)}\Biggr)\frac{y_1^{n+1}}{\sigma_1(\by^n)}\vastr),\\
y_2^{n+1} &= y_2^n + \Delta t\vastl(\Biggl(\biggl(1-\frac1{2\alpha}\biggr) a y_1^n+\biggl(\frac1{2\alpha}\biggr)ay_1^{(2)}\Biggr)\frac{y_1^{n+1}}{\sigma_1(\by^n)} - \Biggl(\biggl(1-\frac1{2\alpha}\biggr) b y_2^n+\biggl(\frac1{2\alpha}\biggr)by_2^{(2)}\Biggr)\frac{y_2^{n+1}}{\sigma_2(\by^n)}\vastr).
\end{split}
\end{equation}
Here we note that $\by^n>\bzero$ implies $\by^{(2)}>\bzero$, see \cite{KM18}, and in addition
\begin{equation}\label{sigma>0}
\sigma_i(\by^n)=\bigl(\bB_\gamma\by^n\bigr)_i^{\frac 1\alpha}\bigl(y_i^n\bigr)^{1-\frac1\alpha}=\bigl(y_i^{(2)}\bigr)^{\frac 1\alpha}\bigl(y_i^n\bigr)^{1-\frac1\alpha}>0.
\end{equation}
The system \eqref{eq:y_n_plus_1_test} can be written as
\begin{multline}\label{eq:y_n_plus_1_A}
\by^{n+1} = \by^n + \Delta t\Biggl( \biggl(1-\frac{1}{2\alpha}\biggr)\bA_P\diag\bigl(\by^{n+1}\bigr)\diag(\bssigma(\by^n))^{-1}\by^n+\biggl(\frac{1}{2\alpha}\biggr)\bA_P\diag\bigl(\by^{n+1}\bigr)\diag(\bssigma(\by^n))^{-1}\by^{(2)}\\
-\biggl(1-\frac{1}{2\alpha}\biggr)\bA_D\diag\bigl(\by^{n+1}\bigr)\diag(\bssigma(\by^n))^{-1}\by^n-\biggl(\frac{1}{2\alpha}\biggr)\bA_D\diag\bigl(\by^{n+1}\bigr)\diag(\bssigma(\by^n))^{-1}\by^{(2)}\Biggr),
\end{multline}
where $\diag(\by)\in\R^{2\times 2}$ denotes the diagonal matrix with $(\diag(\by))_{ii}=y_i$ for $i=1,2$.
As $\diag(\bv)\bw=\diag(\bw)\bv$ and  $\diag(\bv)\diag(\bw)=\diag(\bw)\diag(\bv)$ for all $\bv,\bw\in\R^2$ as well as $\bA=\bA_P-\bA_D$, equation \eqref{eq:y_n_plus_1_A} can be rewritten as 
\begin{equation*}
\by^{n+1} = \by^n + \Delta t\Biggl( \biggl(1-\frac{1}{2\alpha}\biggr)\bA\diag\bigl(\by^{n}\bigr)\diag(\bssigma(\by^n))^{-1}+\biggl(\frac{1}{2\alpha}\biggr)\bA\diag(\by^{(2)})\diag(\bssigma(\by^n))^{-1}\Biggr)\by^{n+1}.
\end{equation*}
Utilizing \eqref{eq:y2_Bdelta_yn},
this can be further simplified to
\begin{equation*}
\by^{n+1} = \by^n + \Delta t\bA\Biggl( \biggl(1-\frac{1}{2\alpha}\biggr)\diag\bigl(\by^{n}\bigr)+\biggl(\frac{1}{2\alpha}\biggr)\diag\bigl(\bB_\gamma\by^{n}\bigr)\Biggr)\diag(\bssigma(\by^n))^{-1}\by^{n+1}.
\end{equation*}
Using $\alpha\diag(\bv)+\beta\diag(\bw)=\diag(\alpha\bv+\beta\bw)$ for all $\alpha,\beta\in\R$ and $\bv,\bw\in\R^2$ together with \eqref{def_C}
we see
\begin{align*}
\by^{n+1} &= \by^n + \Delta t\bA\diag\biggl( \biggl(1-\frac{1}{2\alpha}\biggr)\by^{n}+\biggl(\frac{1}{2\alpha}\biggr)\bB_\gamma\by^n\biggr)\diag(\bssigma(\by^n))^{-1}\by^{n+1}\\&=\by^n + \Delta t\bA\diag\bigl( \bC_\gamma\by^n\bigr)\diag(\bssigma(\by^n))^{-1}\by^{n+1}
\end{align*}
or equivalently
\begin{equation*}
\by^{n+1} = \bM(\by^n)^{-1}\by^n,
\end{equation*}
where
\begin{equation}\label{eq:M_general}
\bM(\by^n)=\bI-\Delta t\bA\diag\bigl( \bC_\gamma\by^n\bigr)\diag(\bssigma(\by^n))^{-1}.
\end{equation}
Hence, we have
\begin{equation*}
\by^{n+1}=\bg(\by^n)
\end{equation*}
with
\[\bg(\by)=\bM(\by)^{-1}\by=\bigl(\bI-\Delta t\bA\diag\bigl( \bC_\gamma\by\bigr)\diag(\bssigma(\by))^{-1}\bigr)^{-1}\by.\]
Next, we want to find an explicit representation of $\bM(\by^n)^{-1}$. 
The diagonal matrix $\diag\bigl(\bC_\gamma\by^n\bigr)$ is nonsingular for all $\by^n>\bzero$, since
\begin{align}
\bC_\gamma\by^n=\biggl(1-\frac{1}{2\alpha}\biggr)\by^n+\biggl(\frac{1}{2\alpha}\biggr)\bB_\gamma\by^n=\biggl(1-\frac{1}{2\alpha}\biggr)\by^n+\biggl(\frac{1}{2\alpha}\biggr)\by^{(2)}>\bzero\label{Cy>0}
\end{align}
holds due to $\alpha\geq\frac12$.
This implies that the matrix $\bA\diag\bigl(\bC_\gamma\by^n\bigr)\diag(\bssigma(\by^n))^{-1}$ has rank one, just like $\bA$.
Hence, according to \cite{MR617892}, the inverse of $\bM(\by^n)$ is given by
\begin{equation}\label{eq:M_inverse}
\bM(\by^n)^{-1}=\bI + \biggl(1-\tr\Bigl(\Delta t\bA\diag\bigl(\bC_\gamma\by^n\bigr)\diag(\bssigma(\by^n))^{-1}\Bigr)\biggr)^{-1}\Delta t\bA\diag\bigl(\bC_\gamma\by^n\bigr)\diag(\bssigma(\by^n))^{-1}.
\end{equation}
Using \eqref{def_h}, we obtain
\begin{equation*}
\diag\bigl(\bC_\gamma\by^n\bigr)\diag(\bssigma(\by^n))^{-1}=
\diag(\btau(\by^n)).
\end{equation*}
and \eqref{eq:M_inverse} becomes
\begin{equation*}
\bM(\by^n)^{-1}=\bI + \frac{\Delta t}{1 + \Delta t(a\tau_1(\by^n)+b\tau_2(\by^n))}\bA\diag(\btau(\by^n)).\qedhere
\end{equation*}
\end{proof}

\begin{rem}
The representation of $\bM(\by^n)$ in \eqref{eq:M_general} is also valid for arbitrary linear positive and conservative PDS $\by'=\bA\by$ with $\bA\in\R^{N\times N}$, where the splitting is given by $\bA=\bA_P-\bA_D$ with $\bA_D=\diag(\bA)$ and $\bA_P=\bA-\bA_D$. The representation of the inverse $\bM(\by^n)^{-1}$ in \eqref{eq:M_inverse} is based on the assumption that $\bA$ has rank one, which must not be satisfied for a general linear positive and conservative PDS. 
\end{rem}

Next, we verify that steady state solutions of \eqref{PDS_test} are indeed fixed points of the MPRK22 schemes. As a consequence every $\by^*=r(b,a)^T$ with $r>0$ is a fixed point of the MPRK22 schemes. This satisfies the requirement of Theorem~\ref{Thm_MPRK_stabil} to have fixed points on a line.
\begin{lem}\label{Lem f(y*)=y*}
Any steady state solution $\by^*>\bzero$ of \eqref{PDS_test} is a fixed point of the MPRK22 schemes, i.\,e.\ a fixed point of the map $\bg$ given in \eqref{def_g(y)}. Moreover, we have
\begin{align*} \bB_\gamma\by^*&=\by^*, & \bC_\gamma\by^*&=\by^*, & \bm{\sigma}(\by^*)&=\by^*, & \tau_1(\by^*)&=\tau_2(\by^*)=1.
\end{align*}
\end{lem}
\begin{proof} 
Let $\by^*$ be a steady state solution of \eqref{PDS_test}, i.\,e. $\bA\by^*=\bzero$. 
At first we show, that $\by^*$ is a fixed point of $\bB_\gamma$ defined in \eqref{def_B_gamma}.
Considering the first term in $\bB_\gamma$ we find $(\bI -\alpha\Delta t\bA)\by^*=\by^*$ and equivalently $\by^*=(\bI-\alpha\Delta t\bA)^{-1}\by^*$. Furthermore, $\bzero=\bA\by^*=\bA_P\by^*-\bA_D\by^*$ is equivalent to $\bA_P\by^*=\bA_D\by^*$. Hence, $(\bI+\alpha\Delta t\bA_D)\by^*=(\bI+\alpha\Delta t\bA_P)\by^*$ or $\by^*=(\bI+\alpha\Delta t\bA_D)^{-1}(\bI+\alpha\Delta t\bA_P)\by^*$.  Altogether, we find
\begin{equation}
\by^{(2)}=\bB_\gamma\by^*=\by^*.\label{By*=y*}
\end{equation}
As a direct consequence, it follows from \eqref{def_C} that
\begin{equation}\label{eq:C_ystern_1}
\bC_\gamma\by^*=\biggl(1-\frac{1}{2\alpha}\biggr)\by^*+\biggl(\frac{1}{2\alpha}\biggr)\bB_\gamma\by^*=\by^*\end{equation}
and in addition, we conclude from \eqref{def_sigma} that
\begin{equation}
\bm{\sigma}(\by^*)=\by^*,\label{sigma(y*)=y*}
\end{equation}
where we used $\by^*>\bzero$. Finally, this implies 
\begin{equation}
\tau_i(\by^*)=\frac{(\bC_\gamma \by^*)_i}{\sigma_i(\by^*)}=1,\label{h(y*)=1}
\end{equation}
for $i=1,2$, and thus, according to \eqref{def_g(y)},
\begin{equation*}
\bg(\by^*)=\biggl(\bI + \frac{\Delta t}{1 + \Delta t(a\tau_1(\by^*)+b\tau_2(\by^*))}\bA\diag\bigl(\btau(\by^*)\bigr)\biggr)\by^*=\by^*+\frac{\Delta t}{1+\Delta t(a+b)}\bA\by^*=\by^*.\qedhere
\end{equation*}
\end{proof}

Based on the knowledge that steady state solutions $\by^*>\bzero$ of \eqref{PDS_test} are fixed points of the MPRK schemes, it is now necessary to compute the eigenvalues of the Jacobian $\bD\bg(\by^*)$. Therefore we have to show that all partial derivatives of $\bg$ exist. The next lemma even states that $\bg\in \mathcal{C}^\infty(\R^2_{>0})$, so that all requirements of Theorem~\ref{Thm_MPRK_stabil} are fulfilled.
\begin{lem}\label{Lem:C2}
The elements $g_i$ of $\bg$ from \eqref{def_g(y)} satisfy $g_i\in \mathcal{C}^\infty (\R^2_{>0})$ for $i\in\{1,2\}$.
\end{lem}
\begin{proof}
Recall from \eqref{def_h} and \eqref{def_g(y)} that 
\begin{equation*}
\bg(\by)=\biggl(\bI + \frac{\Delta t}{1 + \Delta t(a\tau_1(\by)+b\tau_2(\by))}\bA\diag\bigl(\btau(\by)\bigr)\biggr)\by,
\end{equation*}
with $\tau_i(\by)=(\bC_\gamma\by)_i\sigma_i(\by)^{-1}$ for $i=1,2$.
The functions $\sigma_i(\by)=(\bB_\gamma \by)_i^{\frac{1}{\alpha}}(y_i)^{1-\frac{1}{\alpha}}$ and $(\bC_\gamma\by)_i$ are in $\mathcal{C}^\infty (\R^2_{>0})$.
 Furthermore, we know $\sigma_i(\by)>0$ for $\by>\bzero$ according to \eqref{sigma>0}, which yields $\tau_i\in \mathcal{C}^\infty(\R^2_{>0})$ for $i=1,2$ due to the quotient rule. Also $\btau(\by)>\bzero$ for $\by>\bzero$ holds because of \eqref{Cy>0}.
Thus, $1 + \Delta t(a\tau_1(\by^n)+b\tau_2(\by^n))$ is always positive. Consequently even $g_i\in \mathcal{C}^\infty(\R^2_{>0})$ for $i=1,2$.
\end{proof}
Next, we give an explicit representation of the Jacobian $\bD\bg(\by^*)$. As $\bg$ is defined on $\R^2_{>0}$, we have to ensure $\by^*>\bzero$, which requires $a,b>0$. The cases $a=0$ or $b=0$ are special, in that no steady states of \eqref{PDS_test} are contained in $\R^2_{>0}$. These cases will be discussed separately in Section~\ref{b=0_1}.
\begin{lem}\label{Lem Df y=y}
Let $\bg$ be defined by \eqref{def_g(y)} and $\by^*>\bzero$ a steady state of \eqref{PDS_test}, i.\,e. $\bA\by^*=\bzero$. Then, the Jacobian $\bD\bg(\by^*)$ is given by
\begin{equation}
\bD\bg(\by^*)=\bI+\frac{\Delta t}{1+\Delta t(a+b)}\bA\left(\bI+\frac{1}{2\alpha}(\bI-\bB_\gamma)\right).\label{Dg(y*)}
\end{equation}
\end{lem}
\begin{proof}
Setting
\begin{equation*}
\widetilde\btau(\by)=\frac{1}{d(\by)}\Mat{\tau_1(\by)y_1\\\tau_2(\by)y_2},\quad d(\by)=1+\Delta t(a \tau_1(\by)+b \tau_2(\by)),
\end{equation*}
the map $\bg$ from \eqref{def_g(y)}  can be written as
\begin{equation*}
\bg(\by)=\by+\Delta t\bA\widetilde\btau(\by)
\end{equation*}
and consequently we have
\begin{equation}\label{eq:Dg_aux}
\bD\bg(\by)=\bI+\Delta t\bA\bD\widetilde\btau(\by).
\end{equation}
Making use of the notation $\partial_i= \frac{\partial}{\partial y_i}$, we obtain
\begin{align*}
\partial_1\widetilde \tau_1(\by)&=\frac{(\partial_1\tau_1(\by) y_1+\tau_1(\by))d(\by)-\tau_1(\by)y_1\partial_1 d(\by)}{d(\by)^2}, & 
\partial_2\widetilde \tau_1(\by)&=\frac{\partial_2\tau_1(\by) y_1d(\by)-\tau_1(\by)y_1\partial_2 d(\by)}{d(\by)^2},\\
\partial_1\widetilde \tau_2(\by)&=\frac{\partial_1\tau_2(\by) y_2d(\by)-\tau_2(\by)y_2\partial_1 d(\by)}{d(\by)^2}, & 
\partial_2\widetilde \tau_2(\by)&=\frac{(\partial_2\tau_2(\by) y_2+\tau_2(\by))d(\by)-\tau_2(\by)y_2\partial_2 d(\by)}{d(\by)^2}
\end{align*}
and thus
\begin{equation*}
\bD\widetilde\btau(\by) = \frac{1}{d(\by)}\bigl(\diag(\by)\bD\btau(\by) + \diag(\btau(\by))\bigr)-\frac{1}{d(\by)^2}\diag(\btau(\by))\by(\operatorname{grad} d(\by))^T.
\end{equation*}
Inserting this into \eqref{eq:Dg_aux}, the Jacobian of $\bg$ in $\by^*$ is given by
\begin{equation*}
\bD\bg(\by^*) = \bI + \frac{\Delta t}{d(\by^*)}\bA\bigl(\diag(\by^*)\bD\btau(\by^*) + \diag(\btau(\by^*))\bigr)-\frac{\Delta t}{d(\by)^2}\bA\diag(\btau(\by^*))\by^*(\operatorname{grad} d(\by^*))^T.
\end{equation*}
From \eqref{h(y*)=1} we know that $\diag(\btau(\by^*))=\bI$ and together with $\bA\by^*=\bzero$ we find
\begin{equation*}
\bD\bg(\by^*) = \bI + \frac{\Delta t}{d(\by^*)}\bA\bigl(\diag(\by^*)\bD\btau(\by^*) + \bI\bigr).
\end{equation*}
To finish the proof, we have to show $\diag(\by^*)\bD\btau(\by^*)=\frac{1}{2\alpha}(\bI-\bB_\gamma)$.
Therefore, we need to express the partial derivatives of $\tau_i(\by)=(\bC_\gamma\by)_i\sigma_i(\by)^{-1}$ for $i,j=1,2$ in terms of $\bB_\gamma=(b^\gamma_{ij})_{i,j=1,2}$.
The partial derivatives of $\bssigma$ are
\begin{align*}
\partial_1\sigma_1(\by)&=\frac{1}{\alpha}(\bB_\gamma\by)_1^{\frac{1}{\alpha}-1}b^\gamma_{11}y_1^{1-\frac{1}{\alpha}}+(\bB_\gamma\by)_1^{\frac{1}{\alpha}}\biggl(1-\frac{1}{\alpha}\biggr)y_1^{-\frac{1}{\alpha}}, \\ \partial_2\sigma_1(\by)&=\frac{1}{\alpha}(\bB_\gamma\by)_1^{\frac{1}{\alpha}-1}b^\gamma_{12}y_1^{1-\frac{1}{\alpha}},\\
\partial_1\sigma_2(\by)&=\frac{1}{\alpha}(\bB_\gamma\by)_2^{\frac{1}{\alpha}-1}b^\gamma_{21}y_2^{1-\frac{1}{\alpha}}, \\ \partial_2\sigma_2(\by)&=\frac{1}{\alpha}(\bB_\gamma\by)_2^{\frac{1}{\alpha}-1}b^\gamma_{22}y_2^{1-\frac{1}{\alpha}}+(\bB_\gamma\by)_2^{\frac{1}{\alpha}}\biggl(1-\frac{1}{\alpha}\biggr)y_2^{-\frac{1}{\alpha}}
\end{align*}
and as $\bB_\gamma\by^*=\by^*$ we have
\begin{align*}
\partial_1\sigma_1(\by^*)&=\frac{1}{\alpha}b^\gamma_{11}+1-\frac{1}{\alpha}, & \partial_2\sigma_1(\by^*)&=\frac{1}{\alpha}b^\gamma_{12},\\
\partial_1\sigma_2(\by^*)&=\frac{1}{\alpha}b^\gamma_{21}, & \partial_2\sigma_2(\by^*)&=\frac{1}{\alpha}b^\gamma_{22}+1-\frac{1}{\alpha}.
\end{align*}
Furthermore, due to \eqref{def_C} together with $\bC_\gamma=(c^\gamma_{ij})_{i,j=1,2}$ we see
\begin{equation*}
\begin{aligned}
\partial_1 (\bC_\gamma\by)_1 &=c_{11}^\gamma = \biggl(1-\frac1{2\alpha}\biggr) + \frac{1}{2\alpha}b_{11}^\gamma,\quad & 
\partial_2 (\bC_\gamma\by)_1 &=c_{12}^\gamma = \frac{1}{2\alpha}b_{12}^\gamma,\\
\partial_1 (\bC_\gamma\by)_2 &=c_{21}^\gamma = \frac{1}{2\alpha}b_{21}^\gamma, & 
\partial_2 (\bC_\gamma\by)_2 &=c_{22}^\gamma = \biggl(1-\frac1{2\alpha}\biggr) + \frac{1}{2\alpha}b_{22}^\gamma.
\end{aligned}
\end{equation*}
Now we are ready to compute the partial derivatives of $\btau$, whereby we repeatedly use \eqref{eq:C_ystern_1} and \eqref{sigma(y*)=y*}, i.\,e. $\bC_\gamma \by^*=\by^*$ and $\sigma(\by^*)=\by^*$, which yield
\begin{align*}
\partial_1 \tau_1(\by^*)&=\frac{\partial_1(\bC_\gamma\by)_1\sigma_1(\by)-(\bC_\gamma\by)_1\partial_1\sigma_1(\by)}{(\sigma_1(\by))^2}\bigg|_{\by=\by^*}=\frac{\Bigl(1-\frac1{2\alpha} + \frac{1}{2\alpha}b_{11}^\gamma\Bigr)y^*_1-y^*_1\Bigl(\frac{1}{\alpha}b^\gamma_{11}+1-\frac{1}{\alpha}\Bigr)}{(y^*_1)^2}=\frac{1-b_{11}^\gamma}{2\alpha y_1^*},\\
\partial_2\tau_1(\by^*)&=\frac{\partial_2(\bC_\gamma\by)_1\sigma_1(\by)-(\bC_\gamma\by)_1\partial_2\sigma_1(\by)}{(\sigma_1(\by))^2}\bigg|_{\by=\by^*}=\frac{\frac{1}{2\alpha}b^\gamma_{12}y^*_1-y^*_1\frac{1}{\alpha}b^\gamma_{12}}{(y^*_1)^2}=-\frac{b^\gamma_{12}}{2\alpha y^*_1},
\\
\partial_1\tau_2(\by^*)&=\frac{\partial_1(\bC_\gamma\by)_2\sigma_2(\by)-(\bC_\gamma\by)_2\partial_1\sigma_2(\by)}{(\sigma_2(\by))^2}\Bigg|_{\by=\by^*}=\frac{\frac{1}{2\alpha}b_{21}^\gamma y_2^* - y_2^*\frac1{\alpha}b_{21}^\gamma}{(y_2^*)^2}=-\frac{b_{21}^\gamma}{2\alpha y_2^*},\\
\partial_2\tau_2(\by^*)&=\frac{\partial_2(\bC_\gamma\by)_2\sigma_2(\by)-(\bC_\gamma\by)_2\partial_2\sigma_2(\by)}{(\sigma_2(\by))^2}\bigg|_{\by=\by^*}=\frac{\Bigl(1-\frac1{2\alpha}+\frac{1}{2\alpha}b_{22}^\gamma\Bigr)y^*_2-y^*_2\Bigl(\frac{1}{\alpha}b^\gamma_{22}+1-\frac{1}{\alpha}\Bigr)}{(y^*_2)^2}=\frac{1-b^\gamma_{22}}{2\alpha y^*_2}.
\end{align*}
Altogether we see
\begin{equation*}
\bD\btau(\by)=\frac{1}{2\alpha}\diag(\by^*)^{-1}(\bI-\bB_\gamma),
\end{equation*}
which completes the proof.
\end{proof}

We are now in a position to compute the eigenvalues of $\bD\bg(\by^*)$, which are needed to evaluate the stability of the fixed point $\by^*$.
\begin{lem}\label{Lem: Df (1,-1)} 
Every steady state $\by^*>\bzero$ of \eqref{PDS_test} is a non-hyperbolic fixed point of the MPRK22 schemes. In particular, we have
\begin{equation*}
\bD\bg(\by^*)\by^*=\by^*
\end{equation*}
and
\[\bD\bg(\by^*)\bby=R_\gamma(-\Delta t a,-\Delta tb)\bby,\]
where $\bby=(1,-1)^T$ and
\begin{subequations}\label{eq:R_delta}
\begin{equation}
R_1( z_a,z_b)= \frac{2-2\alpha(z_a+z_b)-(z_a+z_b)^2}{2(1-(z_a+z_b))(1-\alpha(z_a+z_b))}\label{def R(alpha,z)}
\end{equation}
as well as
\begin{equation}
R_0(z_a,z_b)=\frac{2-(z_a+z_b)(\frac{z_a}{1-\alpha z_a}+\frac{z_b}{1-\alpha z_b})}{2(1-(z_a+z_b))}.\label{eq:R0}
\end{equation}
\end{subequations}
\end{lem}
\begin{proof}
From the proof of Theorem~\ref{Thm_MPRK_stabil} we know that $\by^*$ is an eigenvector of $\bD\bg(\by^*)$ with associated eigenvalue 1, which can be checked with the straight forward calculation 
\begin{equation*}
\bD\bg(\by^*)\by^*=\by^*+\frac{\Delta t}{1+\Delta t(a+b)}\bA\biggl(\bI+\frac{1}{2\alpha}(\bI-\bB_\gamma)\biggr)\by^*=\by^*+\frac{\Delta t}{1+\Delta t(a+b)}\bA\by^*=\by^*,
\end{equation*}
where we used \eqref{Dg(y*)} and \eqref{By*=y*}.
Hence, $\by^*$ is a non-hyperbolic fixed point of $\bg$.

Furthermore, we know from the proof of Theorem~\ref{Thm_MPRK_stabil} that $\bby=(1,-1)^T$ is another eigenvector of $\bD\bg(\by^*)$ and we need to compute the associated eigenvalue. 

First, we consider the case $\gamma=1$.
The vector $\bby$ is an eigenvector of the matrix $\bA$ from \eqref{PDS_test} with associated eigenvalue $\lambda=-(a+b)$, i.\,e. $\bA\bby=\lambda\bby$.  
If $\gamma=1$,  $\bB_\gamma$ from \eqref{def_B_gamma} becomes $\bB_1= (\bI-\alpha\Delta t\bA)^{-1}$ and 
$(\bI-\alpha\Delta t\bA)\bby = \bby - \alpha\Delta t\lambda \bby = (1-\alpha\Delta t\lambda)\bby$ implies $\bB_1\bby=(1-\alpha\Delta t\lambda)^{-1}\bby$.
Hence, using \eqref{Dg(y*)}, we see that 
\begin{equation*}
\bD\bg(\by^*)\bby=\bby+\frac{\Delta t}{1-\Delta t\lambda}\bA\left(\bby+\frac{1}{2\alpha}(\bI-\bB_1)\bby\right)
=\left(1+ \frac{\Delta t \lambda}{1-\Delta t\lambda}\left(1+\frac{1}{2\alpha}\left(1-\frac{1}{1-\alpha \Delta t\lambda}\right) \right)\right)\bby.
\end{equation*}
Setting $z = \Delta t\lambda$, we find that the eigenvalue associated with $\bby$ for $\gamma = 1$ is given by
\begin{align*}
1+ \frac{z}{1-z}\left(1+\frac{1}{2\alpha}\left(1-\frac{1}{1-\alpha z}\right) \right)
=\frac{2-2\alpha z -z^2}{2(1-z)(1-\alpha z)}.
\end{align*}
With $z_a=-\Delta t a$ and $z_b=-\Delta t b$ we get $z = -\Delta t(a+ b)=z_a+z_b$ and thus
\begin{equation*}
\frac{2-2\alpha z -z^2}{2(1-z)(1-\alpha z)}=\frac{2-2\alpha (z_a+z_b) -(z_a+z_b)^2}{2(1-(z_a+z_b))(1-\alpha (z_a+z_b))}=R_1(z_a,z_b).
\end{equation*}

Next, we consider the case $\gamma=0$ and compute $\bA\bB_0\bby$. From  \eqref{eq:A_splitting}, \eqref{def_B_gamma} and $\bA\bby=\lambda\bby=\lambda\bI\bby$ we see 
\begin{align*}
\bA\bB_0\bby&=\bA(\bI+\Delta t\alpha \bA_D)^{-1}(\bI+\Delta t \alpha \bA_P)\bby=\bA(\bI+\Delta
t \alpha \bA_D)^{-1}(\bI+\Delta t \alpha (\bA+\bA_D))\bby\\
&=\bA(\bI+\Delta
t \alpha \bA_D)^{-1}(\bI+\Delta t \alpha\bA_D+\Delta t\alpha \lambda\bI)\bby
=\lambda\bby+\alpha \Delta t \lambda \bA(\bI+\Delta
t \alpha \bA_D)^{-1}\bby.
\end{align*}
Realizing that $\bby=(1,-1)^T$ is an eigenvector of 
\[\bA(\bI+\Delta
t \alpha \bA_D)^{-1}=
\Mat{-\frac{a}{1+\Delta t\alpha a} & \frac{b}{1+\Delta t\alpha b}\\
\frac{a}{1+\Delta t\alpha a} & -\frac{b}{1+\Delta t\alpha b}}
\] with associated eigenvalue 
\[\mu=-\biggl(\frac{a}{1+\Delta t\alpha a}+\frac{b}{1+\Delta t\alpha b}\biggr),\]
we obtain
\[\bA\bB_0\bby=(\lambda +\alpha\Delta t\lambda \mu)\bby.\]
This together with \eqref{Dg(y*)} shows
\begin{align*}
\bD\bg(\by^*)\bby&=\biggl(\bI+\frac{\Delta t}{1-\Delta t\lambda}\biggl(\bA+\frac{1}{2\alpha}(\bA-\bA\bB_0)\biggr)\biggr)\bby=\biggl(1+\frac{\Delta t\lambda}{1-\Delta t\lambda}+\frac{\Delta t\lambda}{2\alpha(1-\Delta t\lambda)}-\frac{\Delta t\lambda + \alpha\Delta t\lambda\Delta t \mu}{2\alpha(1-\Delta t\lambda)}\biggr)\bby\\
&=\biggl(1+\frac{\Delta t\lambda}{1-\Delta t\lambda}-\frac{\alpha\Delta t\lambda\Delta t \mu}{2\alpha(1-\Delta t\lambda)}\biggr)\bby=\frac{2-\Delta t \lambda\Delta t\mu}{2(1-\Delta t\lambda)}\bby.
\end{align*}
Hence, using $\Delta t\lambda = -\Delta t(a+ b)=z_a+z_b$ as well as $\Delta t\mu =\frac{z_a}{1-\alpha z_a}+\frac{z_b}{1-\alpha z_b}$  the eigenvalue corresponding to $\bby$ for $\gamma = 0$ is
\begin{equation*}
\frac{2- (z_a+z_b)  (\frac{z_a}{1-\alpha z_a}+\frac{z_b}{1-\alpha z_b})}{2(1-(z_a+z_b))}=R_0(z_a,z_b).\qedhere
\end{equation*}
\end{proof}
\begin{rem}
The notation $R_\gamma(z_a,z_b)$ for the eigenvalue in the above lemma was chosen on purpose. If the above analysis is carried out for a Runge--Kutta scheme, the corresponding Jacobian will have the eigenvalues 1 and $R(z_a+z_b)$, in which $R$ denotes the stability function of the Runge--Kutta scheme. In this respect, the function $R_\gamma$ of the MPRK22 schemes plays the same role as the stability function $R$ of a Runge--Kutta scheme. In the following, we refer to $R_\gamma$ as the stability function of the MPRK22 schemes.
\end{rem}
To assess the stability of the non-hyperbolic fixed point $\by^*$ we must investigate the absolute value of the stability function $R_\gamma$ from \eqref{eq:R_delta}.
Theorem~\ref{Thm:_Asym_und_Instabil} states that if $\abs{R_\gamma(-\Delta t a,-\Delta t b)}>1$, the fixed point is unstable. If however $\abs{R_\gamma(-\Delta t a,-\Delta t b)}<1$, we can use Theorem~\ref{Thm_MPRK_stabil} to conclude the stability of $\by^*$.

\begin{lem}\label{Lem: R<1}Let $R_1$ be given by \eqref{def R(alpha,z)}, then the inequality $\lvert R_1(z_a,z_b)\rvert<1$ holds for all $z_a,z_b<0$ and $\alpha\geq \frac{1}{2}$.
\end{lem}
\begin{proof}
First, we realize that $R_1(z_a,z_b)$ only depends on $z_a+z_b$ and hence can be written as $R_1(z)=\frac{2-2\alpha z -z^2}{2(1-z)(1-\alpha z)}$ with $z=z_a+z_b$. To prove the lemma, we need to show $\abs{R_1(z)}<1$ for all $z<0$ and $\alpha\geq \frac12$.
Using $-2z>0$ and $2\alpha\geq 1$, this follows from
\begin{multline*}
\abs{2-2\alpha z-z^2}\leq \abs{2(1-\alpha z)}+\abs{z^2}=2(1-\alpha z)+z^2\\<2(1-\alpha z)+2\alpha z^2-2z=2(1-\alpha z)-2z(1-\alpha z)=2(1-z)(1-\alpha z)=\abs{2(1-z)(1-\alpha z)}.\qedhere
\end{multline*}
\end{proof}
The above lemma states that the stability function of an MPRK22($\alpha$) scheme has absolute value less than one for all time step sizes $\Delta t>0$ and matrix elements $a,b>0$. This allows for the application of  Theorem~\ref{Thm_MPRK_stabil} to conclude that $\by^*$ is indeed a stable fixed point of MPRK22($\alpha$) schemes.
\begin{cor}\label{cor:MPRK22stabil}
Let $\by^*>\bzero$ be an arbitrary steady state of \eqref{PDS_test}.
\begin{enumerate}
\item The MPRK22($\alpha$) schemes are unconditionally stable, in the sense that any $\by^*$ is a stable fixed point of the MPRK22($\alpha$) schemes independent of the time step size $\Delta t$.
\item For every MPRK22($\alpha$) scheme there exists a $\delta>0$, such that $\Vert \by^{0} \Vert_1=\Vert \by^* \Vert_1$ and $\norm{\by^0-\by^*}<\delta$ imply $\by^n\to \by^*$ as $n\to \infty$ independent of the time step size $\Delta t$.
\end{enumerate}
\end{cor}

The situation is different for MPRK22ncs($\alpha$) schemes, for which  the absolute value of the stability function $R_0$ may exceed 1, if $\alpha$ and $\Delta t$ are not chosen properly.
\begin{lem}\label{Prop:R0_properties}
Let $R_0$ be given by \eqref{eq:R0}.
If $\alpha\geq 1$ the inequality
\[\abs{R_0(z_a,z_b)}<1\]
is satisfied for all $z_a,z_b<0$. If $\frac12\leq\alpha<1$, then
\[\abs{R_0(z_a,z_b)}<1\quad\text{for all}\quad z_a<0 \text{  and  } f(z_a)<z_b<0\]
and
\[\abs{R_0(z_a,z_b)}> 1\quad\text{for all}\quad z_a,z_b<0 \text{  and  } z_b<f(z_a),\]
where $f$ is defined by
\begin{equation}\label{eq:fz}
f(\xi)=-\frac{p(\xi)}{2}-\sqrt{\frac{p(\xi)^2}{4}-q(\xi)}
\end{equation}
with
\begin{equation}\label{eq:pq}
p(\xi)=-\frac{2 (1 - \alpha \xi) (2 \alpha + (1 - \alpha) \xi + 
   1)}{(2 \alpha - 1) (1 - \alpha \xi) + \alpha \xi},\quad q(\xi)=\frac{2 (1 - \alpha \xi) (2 - \xi) - 
 \xi^2}{(2 \alpha - 1) (1 - \alpha \xi) + \alpha \xi}.
\end{equation}
\end{lem}
\begin{proof}
First, we show that $R_0$ is strictly increasing with respect to $z_a$ as well as $z_b$. To see this, we consider the partial derivative with respect to $z_a$, which is given by
\begin{align*}
\frac{\partial R_0}{\partial z_a}(z_a,z_b)&=\frac{-\Bigl(\frac{z_a}{1-\alpha z_a}+\frac{z_b}{1-\alpha z_b}\Bigr)- (z_a+z_b)\frac{1}{(1-\alpha z_a)^2}}{2(1-(z_a+z_b))}+\frac{2- (z_a+z_b)  \Bigl(\frac{z_a}{1-\alpha z_a}+\frac{z_b}{1-\alpha z_b}\Bigr)}{2(1-(z_a+z_b))^2}\\
&=\frac{-\biggl(\Bigl(\frac{z_a}{1-\alpha z_a}+\frac{z_b}{1-\alpha z_b}\Bigr)+ \frac{z_a+z_b}{(1-\alpha z_a)^2}\biggr)(1-(z_a+z_b))+2- (z_a+z_b)  \Bigl(\frac{z_a}{1-\alpha z_a}+\frac{z_b}{1-\alpha z_b}\Bigr)}{2(1-(z_a+z_b))^2}\\
&=\frac{-\Bigl(\frac{z_a}{1-\alpha z_a}+\frac{z_b}{1-\alpha z_b}\Bigr)- \frac{z_a+z_b}{(1-\alpha z_a)^2}+\frac{(z_a+z_b)^2}{(1-\alpha z_a)^2}+2}{2(1-(z_a+z_b))^2}
=\frac{-\Bigl(\frac{z_a}{1-\alpha z_a}+\frac{z_b}{1-\alpha z_b}\Bigr)+\frac{(z_a+z_b)(z_a+z_b-1)}{(1-\alpha z_a)^2}+2}{2(1-(z_a+z_b))^2}.
\end{align*}
Since $z_a,z_b<0$ and $\alpha>0$, we find $\frac{\partial R_0}{\partial z_a}(z_a,z_b)>0$, i.\,e. $R_0$ is strictly increasing with respect to $z_a$. In addition, due to the symmetry $R_0(z_a,z_b)=R_0(z_b,z_a)$, it is also strictly in increasing with respect to $z_b$.

Furthermore, we have 
\[\lim_{(z_a,z_b)\to(0,0)} R_0(z_a,z_b)=\lim_{(z_a,z_b)\to(0,0)}\frac{2- (z_a+z_b)  (\frac{z_a}{1-\alpha z_a}+\frac{z_b}{1-\alpha z_b})}{2(1-(z_a+z_b))}=1\]
and due to the monotonicity $R_0(z_a,z_b)<1$ for all $z_a,z_b<0$. To compute a lower bound for $R_0$, we rewrite $R_0$ in the form
\begin{equation}\label{eq:R0_aux}
R_0(z_a,z_b)
=\frac{2(1-\alpha z_a)(1-\alpha z_b)-(z_a+z_b)(z_a(1-\alpha z_b)+z_b(1-\alpha z_a))}{2(1-(z_a+z_b))(1-\alpha z_a)(1-\alpha z_b)},
\end{equation}
which can also be written as
\begin{equation*}
R_0(z_a,z_b)=\frac{2(1-\alpha z_a)(\frac{1}{z_b^2}-\frac{\alpha} {z_b})-(\frac{z_a}{z_b}+1)(z_a(\frac{1}{z_b}-\alpha)+(1-\alpha z_a))}{2(\frac{1}{z_b}-(\frac{z_a}{z_b}+1))(1-\alpha z_a)(\frac{1}{z_b}-\alpha)}.
\end{equation*}
Now we see
\begin{equation*}
\lim_{z_b\to -\infty}R_0(z_a,z_b)=\frac{-(-\alpha z_a+1-\alpha z_a)}{-2(1-\alpha z_a)(-\alpha)}=-\frac{1-2\alpha z_a}{2\alpha(1-\alpha z_a)}
\end{equation*}
and 
\begin{equation*}
\lim_{z_a\to -\infty}\lim_{z_b\to -\infty}R_0(z_a,z_b)=\lim_{z_a\to -\infty}-\frac{1-2\alpha z_a}{2\alpha(1-\alpha z_a)}=\lim_{z_a\to -\infty}-\frac{\frac{1}{z_a}-2\alpha }{2\alpha(\frac{1}{z_a}-\alpha)}=-\frac{1}{\alpha}.
\end{equation*}
Since $R_0$ is continuous, we even have $\lim_{(z_a,z_b)\to(-\infty,-\infty)}R_0(z_a,z_b)=-\frac{1}{\alpha}$ and since $R_0$ is strictly increasing in $z_a$ and $z_b$, we obtain $-\frac{1}\alpha < R_0(z_a,z_b)$. Altogether, we know $-\frac{1}{\alpha}<R_0(z_a,z_b)<1$ for all $z_a,z_b<0$, which implies \[\abs{R_0(z_a,z_b)}<1\quad\text{for all}\quad z_a,z_b<0\text{   and   }\alpha\geq 1.\] 
However, if $\alpha<1$ then $\lim_{(z_a,z_b)\to(-\infty,\infty)}R_0(z_a,z_b)=-\frac 1\alpha<-1$ and hence there exist $ z_a$ and $ z_b$ for which $\abs{R_0(z_a,z_b)}>1$. 

Due to the monotonicity and continuity of $R_0$ there is exactly one $z_b$ for a given $z_a$, such that $R_0( z_a, z_b)=-1$. To find this $z_b$ we need to solve the equation $R_0(z_a,z_b)=-1$ for $z_b$. According to \eqref{eq:R0_aux}, this is equivalent to
\begin{equation*}
\frac{2(1-\alpha z_a)(1-\alpha z_b)-(z_a+z_b)(z_a(1-\alpha z_b)+z_b(1-\alpha z_a))}{2(1-(z_a+z_b))(1-\alpha z_a)(1-\alpha z_b)}=-1
\end{equation*}
or 
\begin{equation*}
2(1-\alpha z_a)(1-\alpha z_b)-(z_a+z_b)(z_a(1-\alpha z_b)+z_b(1-\alpha z_a))=-2(1-(z_a+z_b))(1-\alpha z_a)(1-\alpha z_b).
\end{equation*}
A technical computation reveals that this equation can be rewritten in the form
\begin{equation*}
\bigl((2 \alpha - 1) (1 - \alpha z_a) + \alpha z_a\bigr)z_b^2-\bigl(2 (1 - \alpha z_a) (2 \alpha + (1 - \alpha) z_a + 
   1)\bigr)z_b + 2 (1 - \alpha z_a) (2 - z_a) - 
    z_a^2=0
\end{equation*}
and solving for $z_b$ shows $z_b=f(z_a)$, with $f$ defined in \eqref{eq:fz}, \eqref{eq:pq}.
Due to the monotonicity of $R_0$, we have $R_0(z_a,z_b)<-1$ for all $z_a<0$ and $z_b<f(z_a)$ and $-1<R_0(z_a,z_b)<1$ for all $z_a<0$ and $f(z_a)<z_b<0$. Hence, we have also proven the statement of the lemma for the case $\frac12\leq\alpha<1$.
\end{proof}
An immediate consequence of the above lemma in combination with Theorem~\ref{Thm:_Asym_und_Instabil} and Theorem~\ref{Thm_MPRK_stabil} is the following corollary.
\begin{cor}\label{cor:unstable}
Let $\by^*>\bzero$ be an arbitrary steady state of \eqref{PDS_test} and  $\Delta t_\alpha^*>0$ be the unique solution of $f(-\Delta t_\alpha^* a)=-\Delta t_\alpha^* b$, where $f$ is defined in \eqref{eq:fz}.

\begin{enumerate}
\item 
If $\alpha\geq 1$ or $\Delta t<\Delta t_\alpha^*$, then the MPRK22ncs($\alpha$) schemes are stable, in the sense that any $\by^*$ is a stable fixed point of the MPRK22ncs($\alpha$) schemes.
\item For every MPRK22ncs($\alpha$) scheme with $\alpha\geq 1$ or $\Delta t<\Delta t_\alpha^*$ there exists a $\delta>0$, such that $\Vert \by^{0} \Vert_1=\Vert \by^* \Vert_1$ and $\norm{\by^0-\by^*}<\delta$ imply $\by^n\to \by^*$ as $n\to \infty$.
\item  If $\Delta t>\Delta t_\alpha^*$ and $\frac12\leq\alpha < 1$, then the MPRK22ncs($\alpha$) schemes are unstable, in the sense that every steady state $\by^*$ of \eqref{PDS_test} is an unstable fixed point of the MPRK22ncs($\alpha$) schemes.
\end{enumerate}
\end{cor}

According to part a) of Corollary \ref{cor:unstable}, we can define the stability regions for MPRK22ncs($\alpha$) schemes with $\frac12\leq \alpha<1$ as the set of points $S(\alpha) \subset \R^- \times \R^-$ lying above the graph of the function $f$ from \eqref{eq:fz}, see Figure~\ref{StabilityRegion}. We also want to pay attention to the fact that the MPRK22ncs($\alpha$) schemes are stable for $\alpha\geq 1$, which coincides with the expansion of the stability region $S(\alpha)$ for $\alpha\to 1$ that can be observed within Figure~\ref{StabilityRegion}.


\begin{figure}[!h]
\begin{subfigure}[t]{0.495\textwidth}
\includegraphics[width=\textwidth]{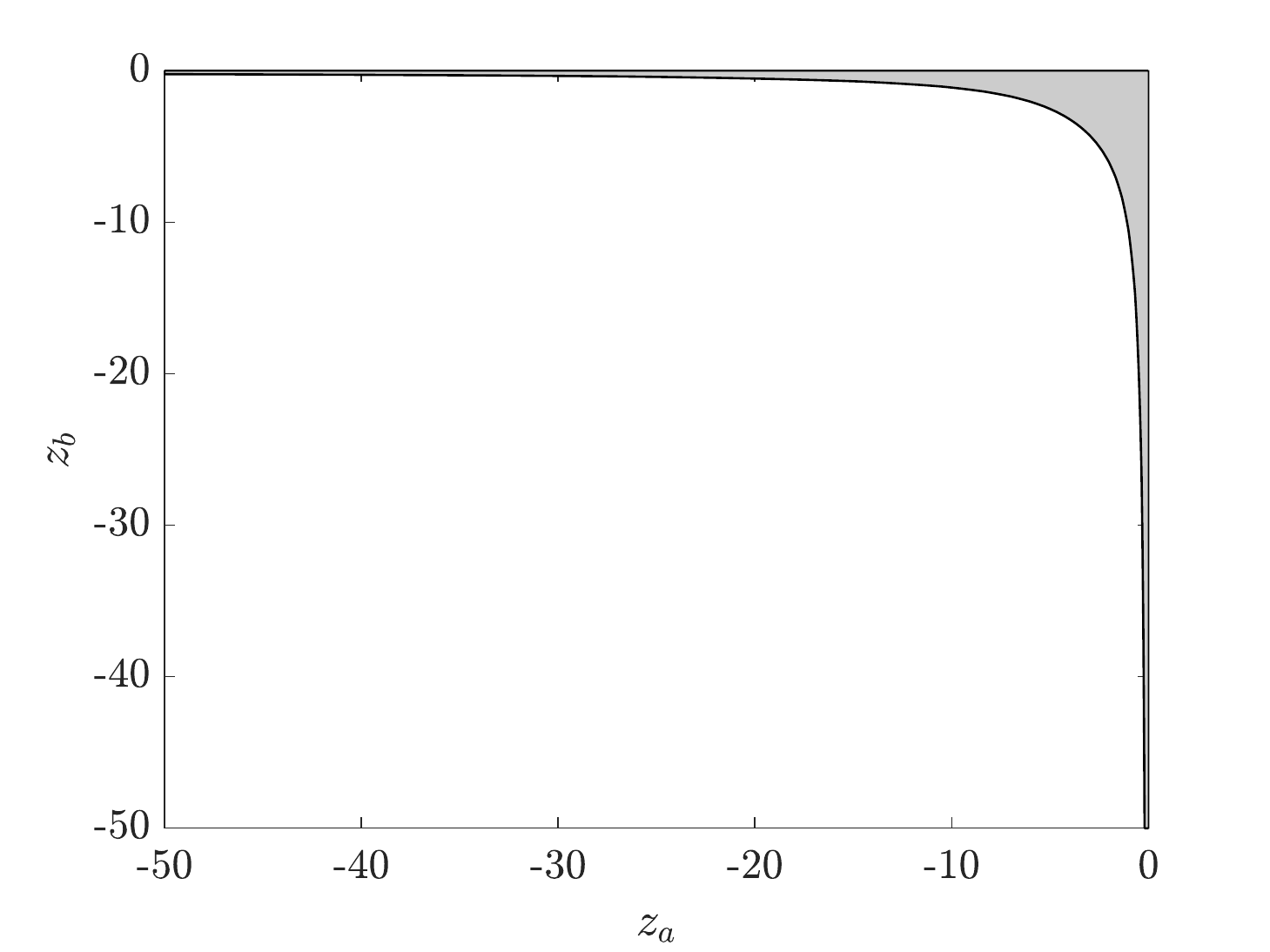}
\subcaption{$\alpha=0.5$}
\end{subfigure}
\begin{subfigure}[t]{0.495\textwidth}
\includegraphics[width=\textwidth]{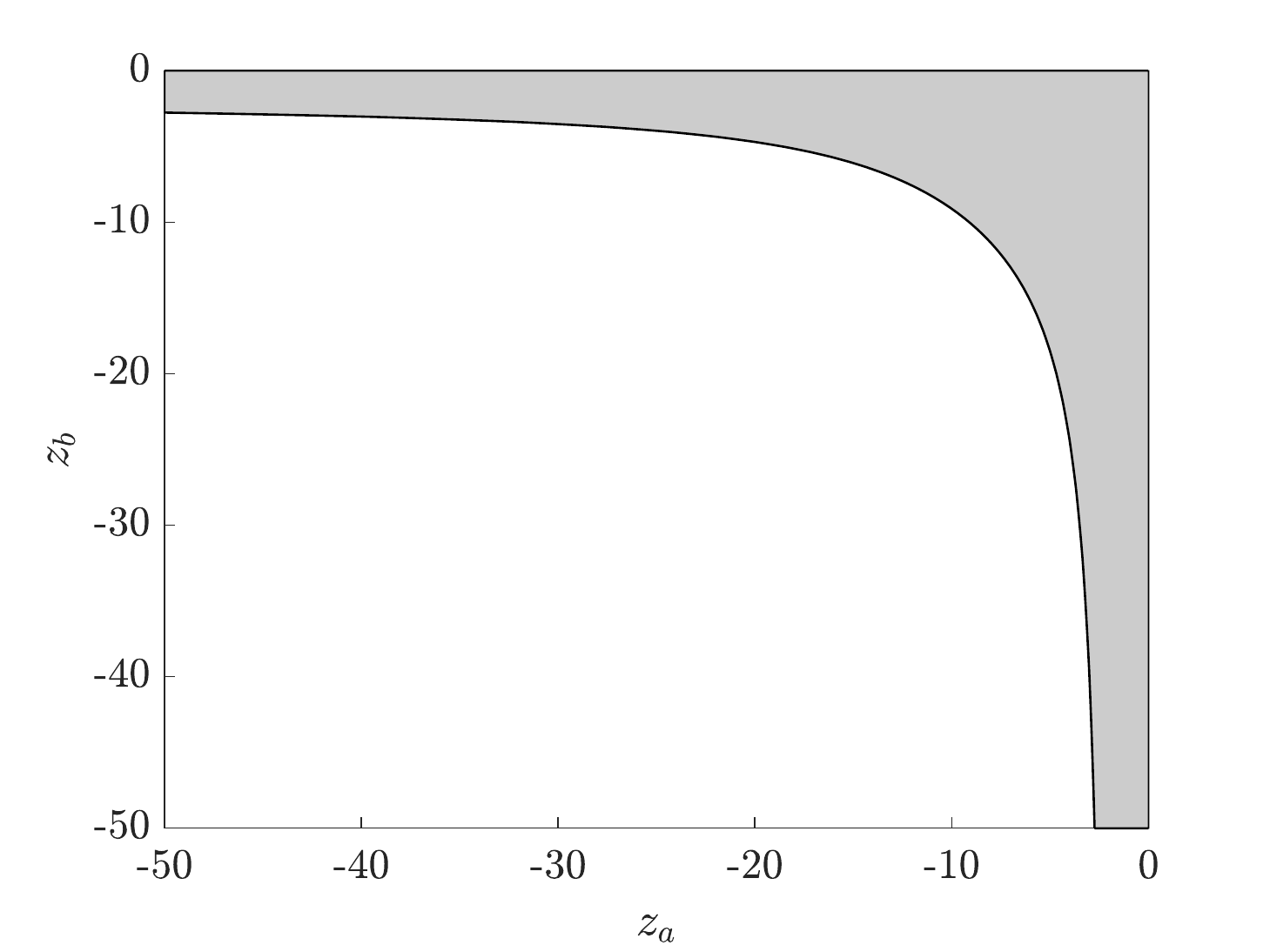}
\subcaption{$\alpha=0.8$}
\end{subfigure}\\
\begin{subfigure}[t]{0.495\textwidth}
\includegraphics[width=\textwidth]{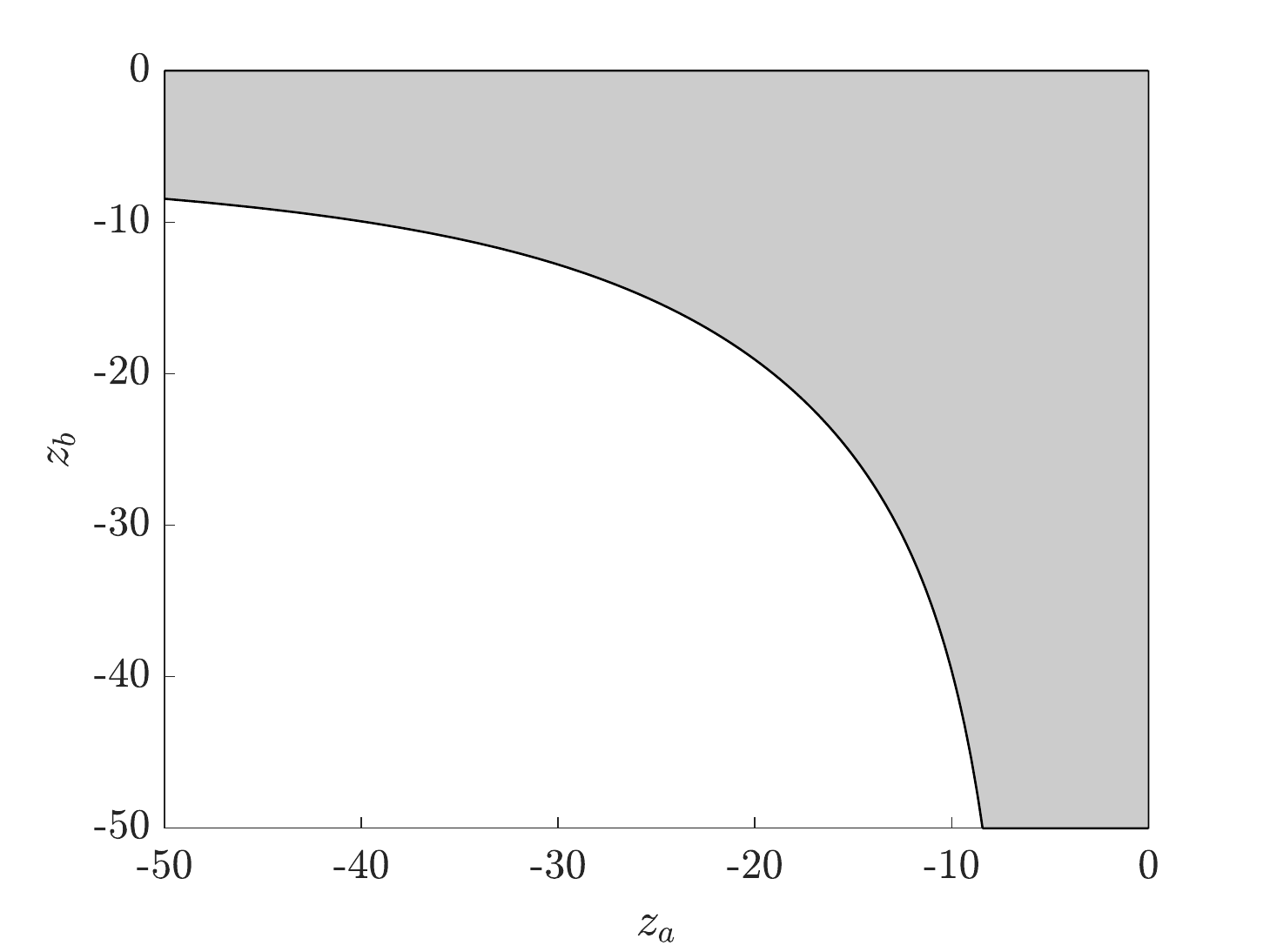}
\subcaption{$\alpha=0.9$}
\end{subfigure}
\begin{subfigure}[t]{0.495\textwidth}
\includegraphics[width=\textwidth]{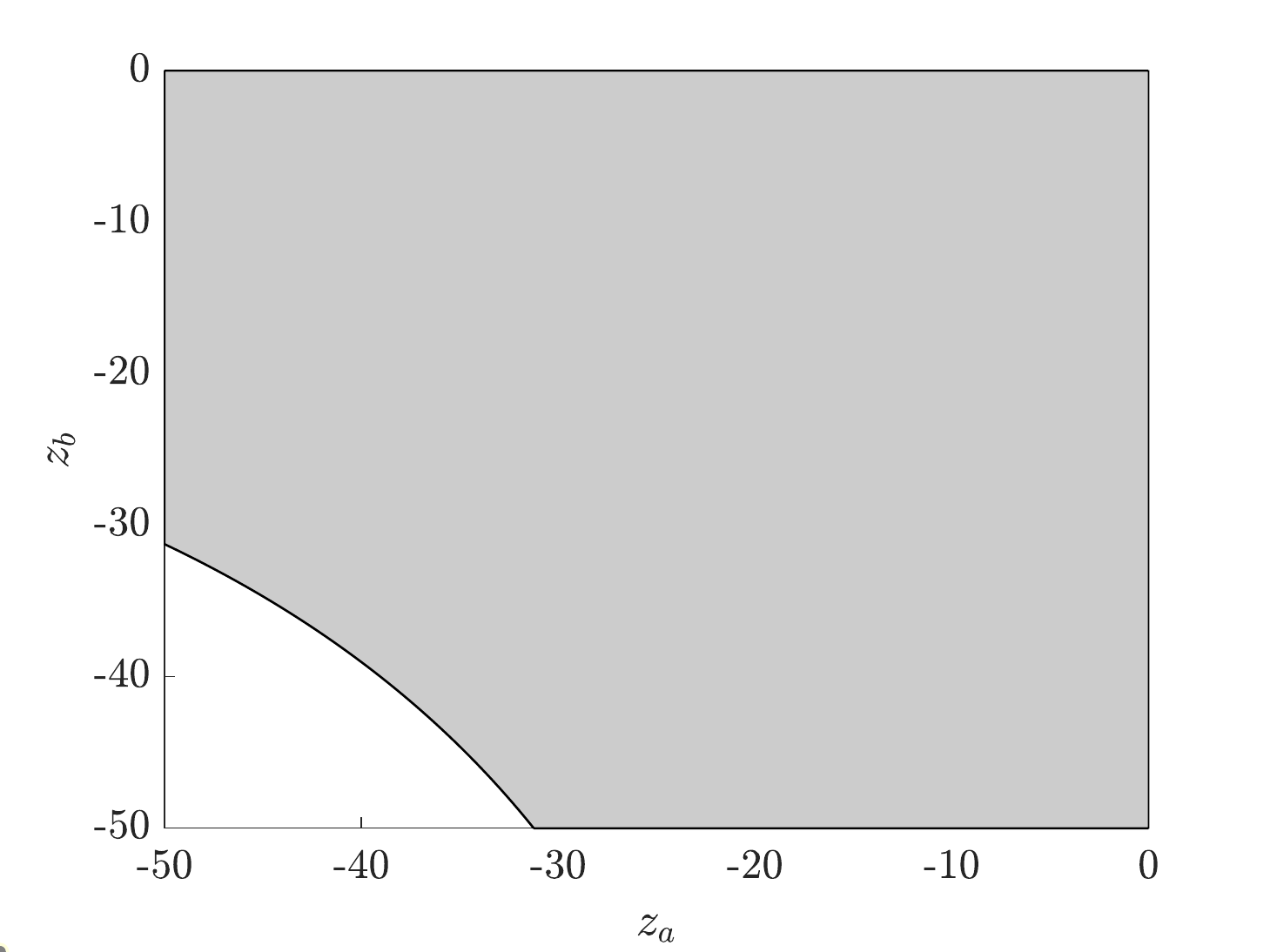}
\subcaption{$\alpha=0.95$}
\end{subfigure}
\caption{For each $\alpha$ the grey area indicates the stability region $S(\alpha)$ of the MPRK22ncs$(\alpha)$ scheme for $z_a,z_b\in[-50,0)$.}\label{StabilityRegion}
\end{figure}


\subsection{The cases $a=0$ or $b = 0$}\label{b=0_1}
As mentioned before, the cases $a=0$ and $b=0$ are special, in that no steady states of \eqref{PDS_test} are contained in $\R^2_{>0}$. 
In the following, we only discuss the case $b=0$, i.\,e.
\begin{equation}\label{eq:system_b=0}
y_1'=-a y_1,\quad y_2=a y_1,\quad a>0,
\end{equation} 
since interchanging the roles of $y_1$ and $y_2$ leads to the case $a=0$. 
This system \eqref{PDS_test} has the steady states $\by^*=(0,y_2^*)\in\R^2$ with $y_2^*\geq 0$ and we show that 
the MPRK22 iterations with an arbitrary initial condition $\by^0>\bzero$ converge to the steady state $\by^*=(0,y_1^0+y_2^0)$, which is the steady state of the continuous problem, see \eqref{exact_sol}. 
Since the MPRK22 schemes are conservative, it is sufficient to prove $y_1^n\to 0$ as $n\to \infty$, since this directly implies $y_2^n\to y_1^0+y_2^0$ as $n\to \infty$. 

Application of the MPRK22 schemes \eqref{eq:MPRK22b} to \eqref{eq:system_b=0} results in 
\begin{subequations}
	\begin{align}
y_1^{(1)} &= y_1^n,\\
y_1^{(2)} &= y_1^n - \alpha\Delta t a y_1^{(1)}\frac{y_1^{(2)}}{y_1^{(1)}}= y_1^n - \alpha\Delta t a y_1^{(2)},\label{eq:y2_b=0}\\
y_1^{n+1} &= y_1^n - \Delta t
	\Biggl(\biggl(1-\frac1{2\alpha}\biggr) a y_1^{(1)}+ \frac1{2\alpha} a y_1^{(2)}\Biggr)\frac{y_1^{n+1}}{(y_1^{(2)})^{\frac{1}{\alpha}}(y_1^{(1)})^{1-\frac{1}{\alpha}}}.	\label{eq:yneu_b=0}
	\end{align}
\end{subequations}
We note that there is no dependence on $\gamma$, hence, the following applies to both MPRK22($\alpha$) and MPRK22ncs($\alpha$) schemes.
Solving \eqref{eq:y2_b=0} for $y_1^{(2)}$ yields
\begin{equation*}
y_1^{(2)}=\frac{y_1^n}{1+\alpha\Delta t a}.
\end{equation*}
and inserting this into \eqref{eq:yneu_b=0} shows
\begin{align*}
y_1^{n+1} &= y_1^n - \Delta t a
\Biggl(\biggl(1-\frac1{2\alpha}\biggr) y_1^{n}+ \frac1{2\alpha} \cdot \frac{y_1^n}{1+\alpha\Delta t a}\Biggr)\frac{(1+\alpha\Delta t a)^{\frac{1}{\alpha}}}{(y_1^n)^\frac{1}{\alpha}}\cdot \frac{y_1^{n+1}}{(y_1^{n})^{1-\frac{1}{\alpha}}}\\
&= y_1^n - \Delta t a
\Biggl(\biggl(1-\frac1{2\alpha}\biggr) +  \frac{1}{2\alpha(1+\alpha\Delta t a)}\Biggr)(1+\alpha\Delta t a)^{\frac{1}{\alpha}}y_1^{n+1}\\
&= y_1^n -\frac{ \Delta t a\bigl(
1+(\alpha-\frac12)\Delta t a\bigr)}{(1+\alpha\Delta t a)^{1-\frac{1}{\alpha}}}y_1^{n+1}.
\end{align*}
Solving for $y_1^{n+1}$ leads to
\begin{equation*}
y_1^{n+1}=R(-\Delta t a)y_1^n,
\end{equation*}
where $R$ is defined by
\begin{equation*}
R(z)=\biggl(1-\frac{ z(1-(\alpha-\frac12)z)}{(1-\alpha z)^{1-\frac{1}{\alpha}}}\biggr)^{-1}=\frac{(1-\alpha z)^{1-\frac{1}{\alpha}}}{(1-\alpha z)^{1-\frac{1}{\alpha}}-z\bigl(1-(\alpha-\frac12)z\bigr)}.
\end{equation*}
For $z<0$ and $\alpha\geq \frac12$ we have $1-\alpha z>0$ and $-z(1-(\alpha-\frac12)z)>0$, which implies $0<R(z)<1$.
Hence, the  MPRK22($\alpha$) and MPRK22ncs($\alpha$) schemes  converge montonically towards the correct steady state along the line $y_1+y_2=y_1^0+y_2^0$, just like the solution of the continuous problem. 

As we have seen, there is no difference between MPRK22($\alpha$) and MPRK22ncs($\alpha$) schemes in the cases $a=0$ or $b=0$, which might lead to the conclusion that both schemes have equal stability properties in general. In particular, since \eqref{eq:system_b=0} is a straightforward extension of Dahlquist's equation to positive and conservative PDS. But the analysis on the general system \eqref{PDS_test} in Section~\ref{sec:stab_mprk22} shows significant differences with respect to the stability of MPRK22($\alpha$) and MPRK22ncs($\alpha$) schemes. Hence, it is insufficient to use \eqref{eq:system_b=0} to evaluate the stability of schemes which do not belong to the class of general linear methods. 
 
\section{Numerical Experiments}\label{Num Tests}
In this section we perform numerical experiments to confirm the stability properties of both MPRK22($\alpha$) and MPRK22ncs($\alpha$) schemes. For this purpose, we consider the test equation \eqref{PDS_test} for the case $a = b=25$. 
In particular, we consider the initial value problem 
\begin{equation}\label{test_equation}
	\mbfy'= \begin{pmatrix*}[r]
		-25 &25\\ 25 &-25
	\end{pmatrix*} \mbfy \quad \text{with} \quad  \by^0=\begin{pmatrix*}0.998\\0.002\end{pmatrix*}.
\end{equation}
The nonzero eigenvalue is $\lambda = -(a+b) = -50$ and the analytic solution is given by
\begin{equation}\label{eq:Test_exact}
	\mbfy(t)=\frac{y_1^0+y_2^0}{a+b}\begin{pmatrix}
		b\\a 
	\end{pmatrix}+\frac{a y_1^0-by_2^0}{a+b}\Vec{1\\ -1}e^{\lambda t}
=  \frac12\begin{pmatrix}
	1\\1 
\end{pmatrix}+0.498 \Vec{1\\ -1}e^{-50 t}.
\end{equation}
As depicted in Figure~\ref{Fig:exact_sol}, it can be observed that the equilibrium state 
\[
\by^*=\frac{y_1^0+y_2^0}{a+b}\begin{pmatrix}
	b\\a
\end{pmatrix}=\frac12 \begin{pmatrix}
1\\1
\end{pmatrix}
\]
is approximately already reached at time $t = 0.1$.
\begin{figure}
\centering
\includegraphics[width=0.5\textwidth]{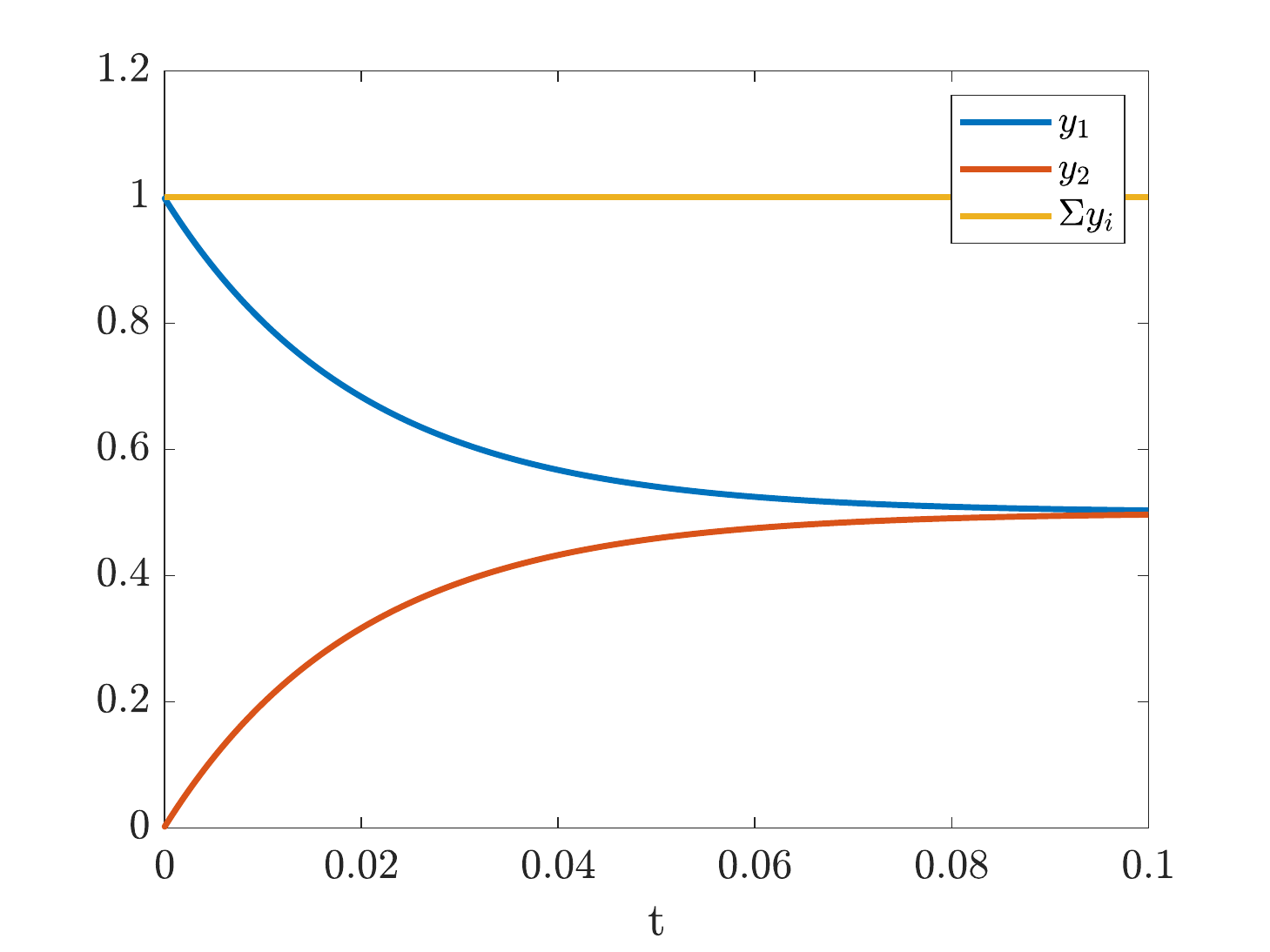}
\caption{Exact solution \eqref{eq:Test_exact} of the test problem \eqref{test_equation}.}\label{Fig:exact_sol}
\end{figure}

To verify the theoretical statements of Corollaries~\ref{cor:MPRK22stabil}~and~\ref{cor:unstable}, we investigate at first the case of $\alpha \geq 1$, for which we have proven that MPRK22($\alpha$) as well as MPRK22ncs($\alpha$) schemes are stable and locally convergent to the correct steady state of \eqref{PDS_test}, irrespectively of the chosen time step size. To illustrate these positive properties of the methods, we choose $\Delta t = 4$ and $\Delta t = 20$, in order to use very large time step sizes compared to the time scale of the exact solution of $0.1$ as mentioned above. Numerical approximations obtain with these time step sizes can be seen in Figure~\ref{Fig:MPRK22_all_Delta}. Thereby, the predicted stability of the schemes is clearly demonstrated. In the case of $\Delta t = 4$, the iterates of both MPRK22(1) and MPRK22ncs(1) converge to the steady state of the exact solution shown as a dashed lines. It is worth mentioning that the result of MPRK22(1) shows significantly less oscillations compared to the result of MPRK22ncs(1) and is also very close to the steady state already for t = 30, while MPRK22ncs(1) requires much more iterations to approach the steady state.
The same behavior can be observed for $\Delta t= 20$, for which both schemes are stable and convergent to the steady state. Nevertheless, even in this case the MPRK22(2) shows smaller amplitudes and less oscillation compared to MPRK22ncs(2). In summary, the numerical experiments verify the theoretical results for both MPRK22($\alpha$) and MPRK22ncs($\alpha$) schemes in the case of $\alpha \geq 1$. 

\begin{figure}[!h]
\begin{subfigure}[t]{0.495\textwidth}
\includegraphics[width=\textwidth]{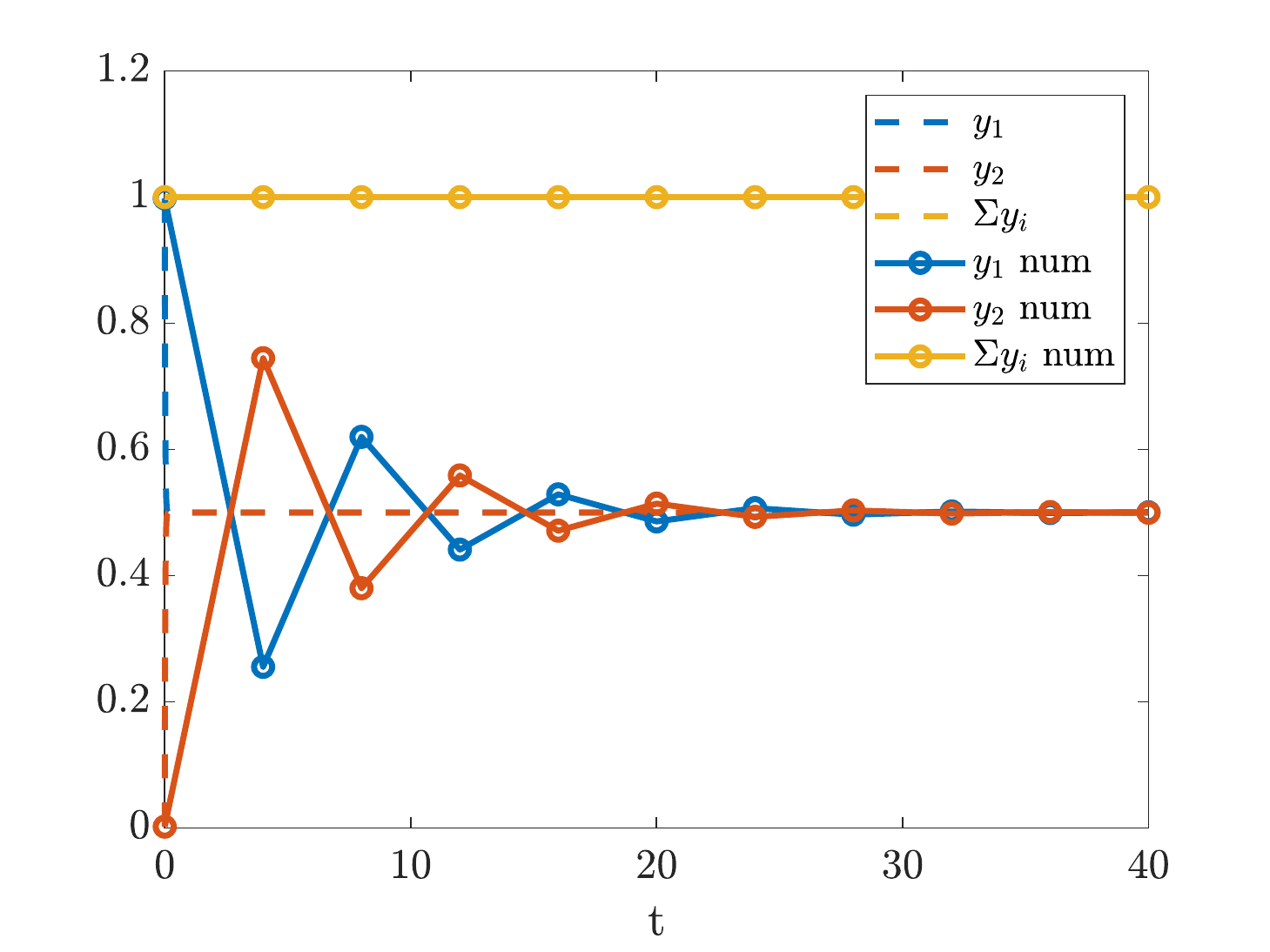}
\subcaption{ MPRK22(1) with $\Delta t=4$}
\end{subfigure}
\begin{subfigure}[t]{0.495\textwidth}
\includegraphics[width=\textwidth]{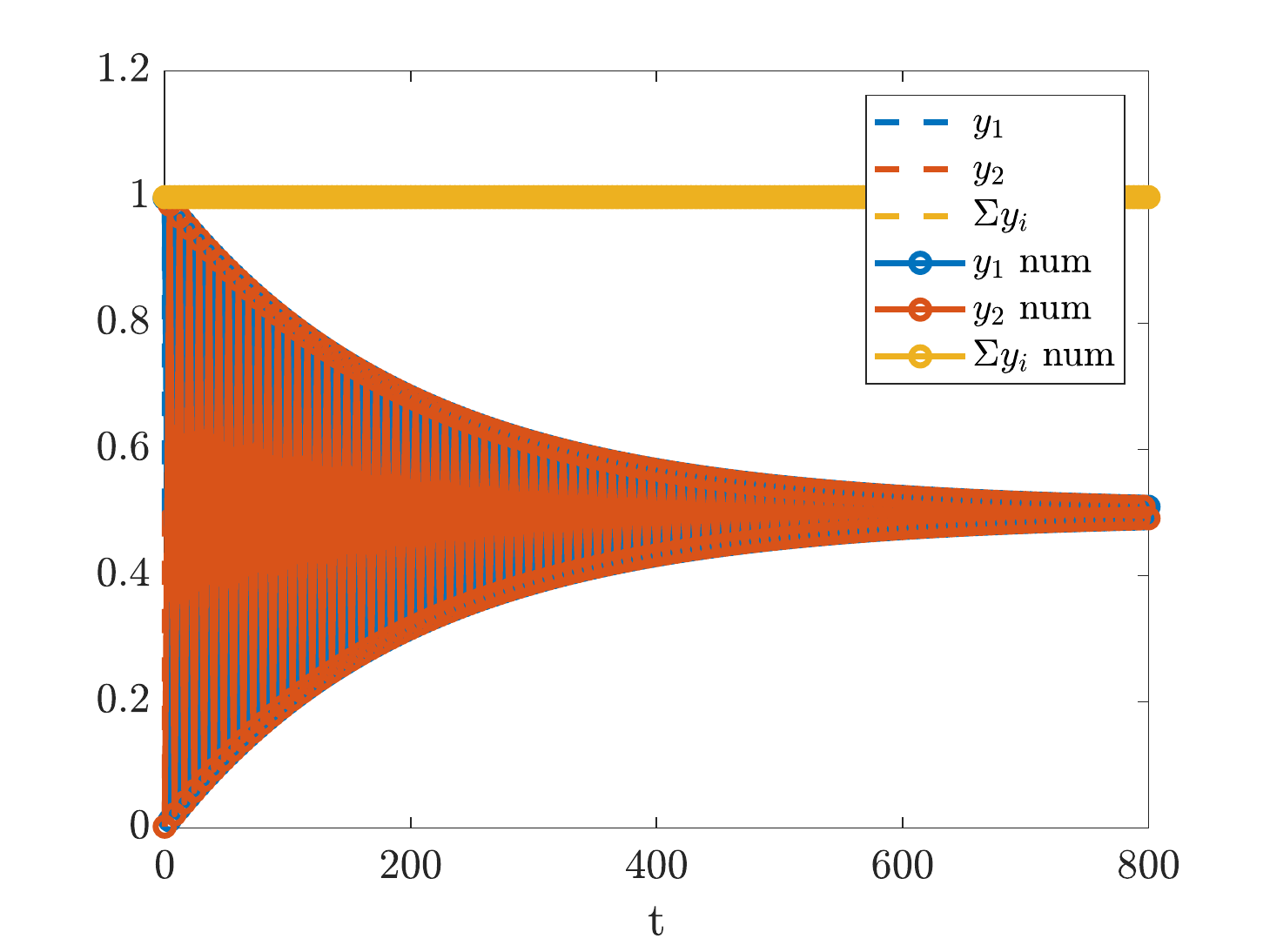}
\subcaption{MPRK22ncs(1) with $\Delta t=4$}
\end{subfigure}\\

\begin{subfigure}[t]{0.495\textwidth}
\includegraphics[width=\textwidth]{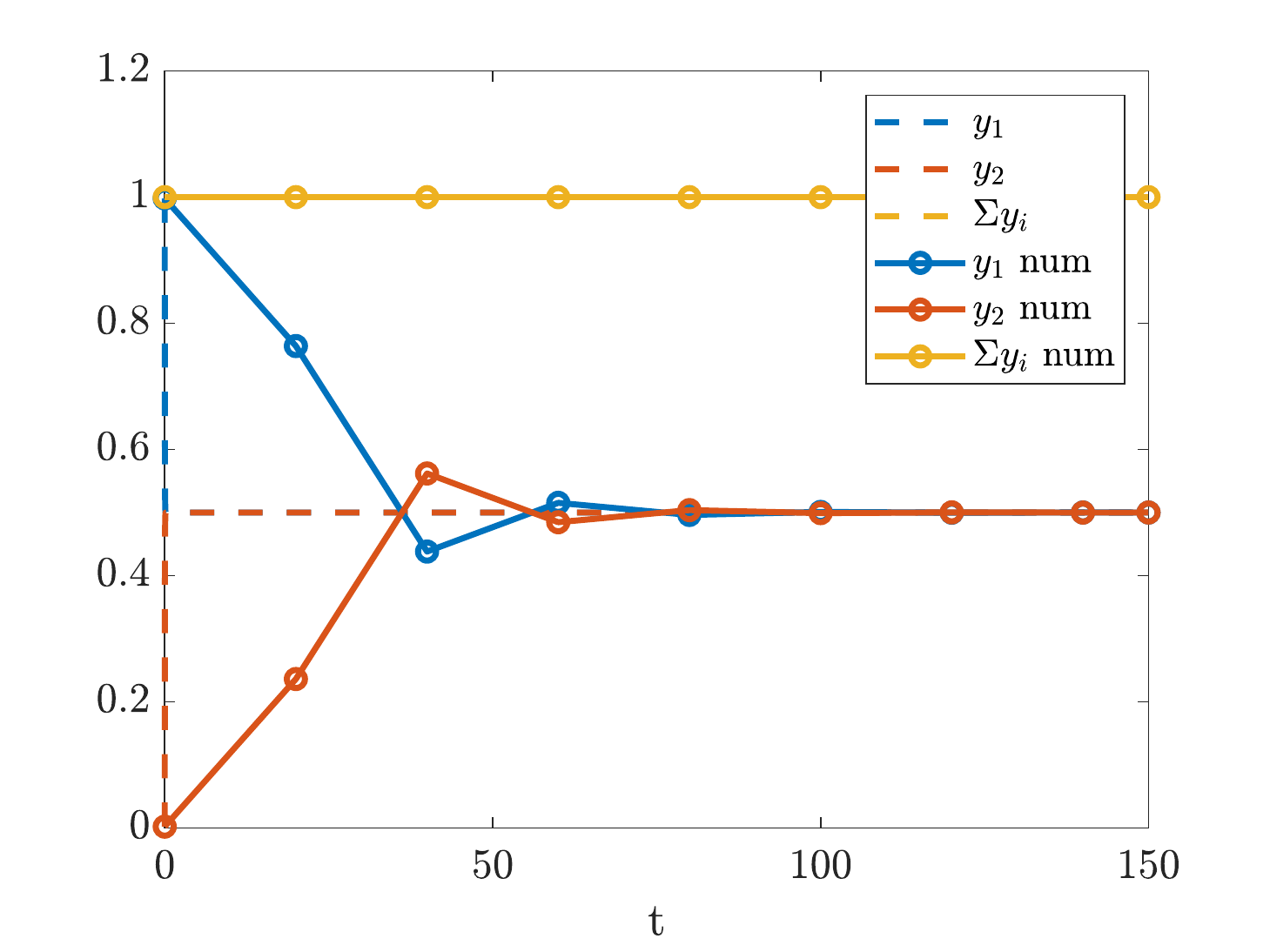}
\subcaption{MPRK22(2) with $\Delta t=20$}
\end{subfigure}
\begin{subfigure}[t]{0.495\textwidth}
\includegraphics[width=\textwidth]{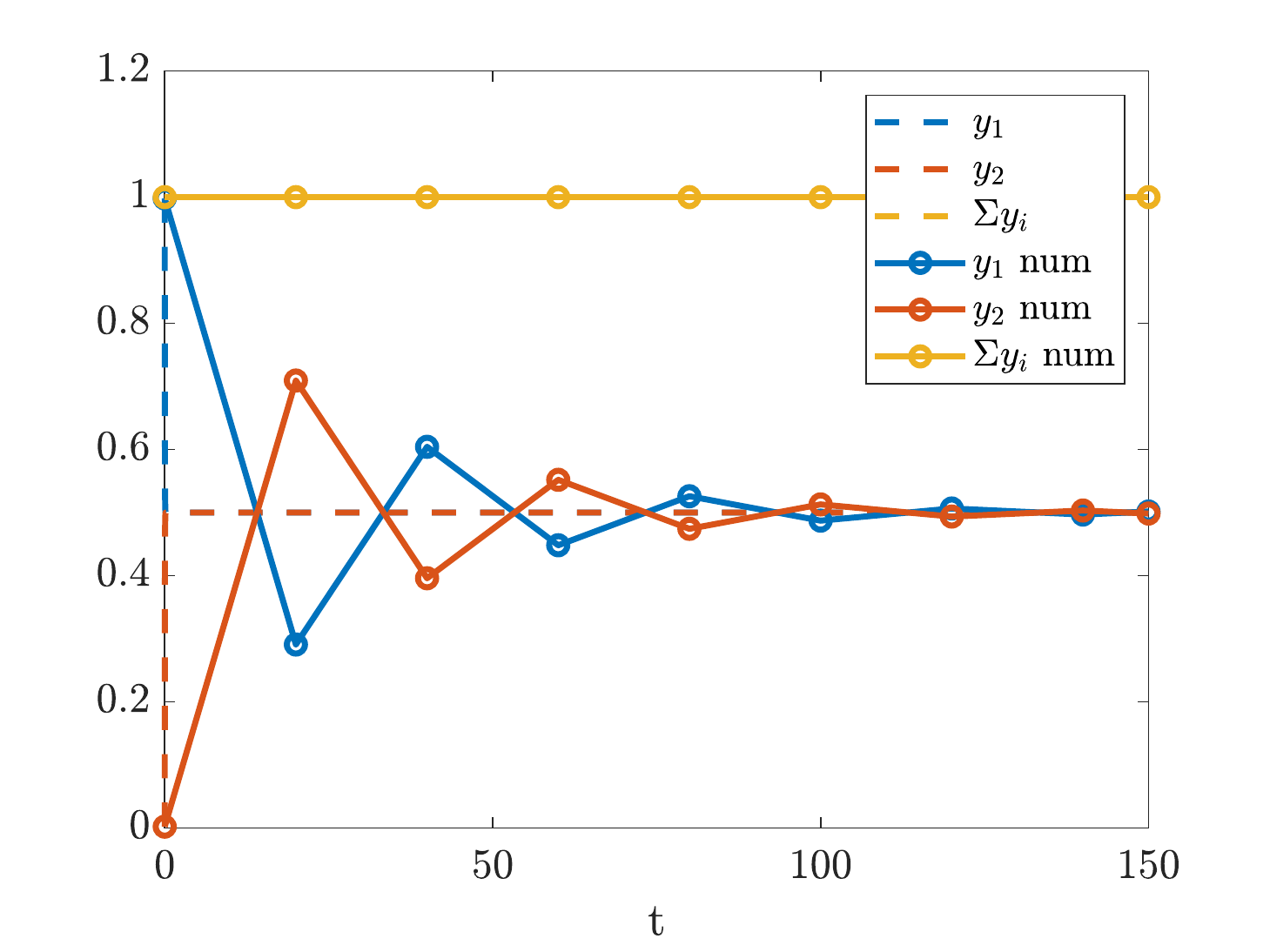}
\subcaption{MPRK22ncs(2) with $\Delta t=20$}
\end{subfigure}
\caption{Numerical approximations of \eqref{test_equation}. The dashed lines indicate the exact solution \eqref{eq:Test_exact}.}
\label{Fig:MPRK22_all_Delta}
\end{figure}

The choice of parameters $\frac12 \leq \alpha < 1$ requires a more differentiated consideration since in this case MPRK22($\alpha$) schemes are stable due to Corollary~\ref{cor:MPRK22stabil} while MPRK22ncs($\alpha$) schemes postulate stability conditions with respect to the time step size according to Corollary~\ref{cor:unstable}. In the following we focus on $\alpha = 0.5$ and $\alpha = 0.8$. To demonstrate the different stability behavior, we choose points of the form
\[
  \mbfz(\Delta t) = \begin{pmatrix}
  	z_a(\Delta t)\\z_b(\Delta t)
  \end{pmatrix}=-\Delta t \begin{pmatrix}
  a\\b
\end{pmatrix}=-\Delta t \begin{pmatrix}
25\\25
\end{pmatrix}
\]
inside and outside the stability domain. The location of the points $\mbfz(\Delta t)$ is visualized for $\alpha = 0.5$ as well as $\alpha = 0.8$ by a red line in Figure~\ref{Fig:red_line}. Following Corollary~\ref{cor:unstable} the point $\mbfz(\Delta t_\alpha^*)$ lies on the boundary of the stability region and straightforward calculations yield $\Delta t_\alpha^*=(\sqrt{17}+3)/50\approx 0.14$ for $\alpha = 0.5$ and $\Delta t_\alpha^*=(\sqrt{101}+9)/50\approx 0.38$ for $\alpha = 0.8$, respectively. The two remaining points within the figures result in each case from the choices $\Delta t_1 = \Delta t_\alpha^* - 10^{-1}$ and $\Delta t_2 = \Delta t_\alpha^* + 10^{-1}$. 

\begin{figure}[!h]
\begin{subfigure}[t]{0.495\textwidth}
\includegraphics[width=\textwidth]{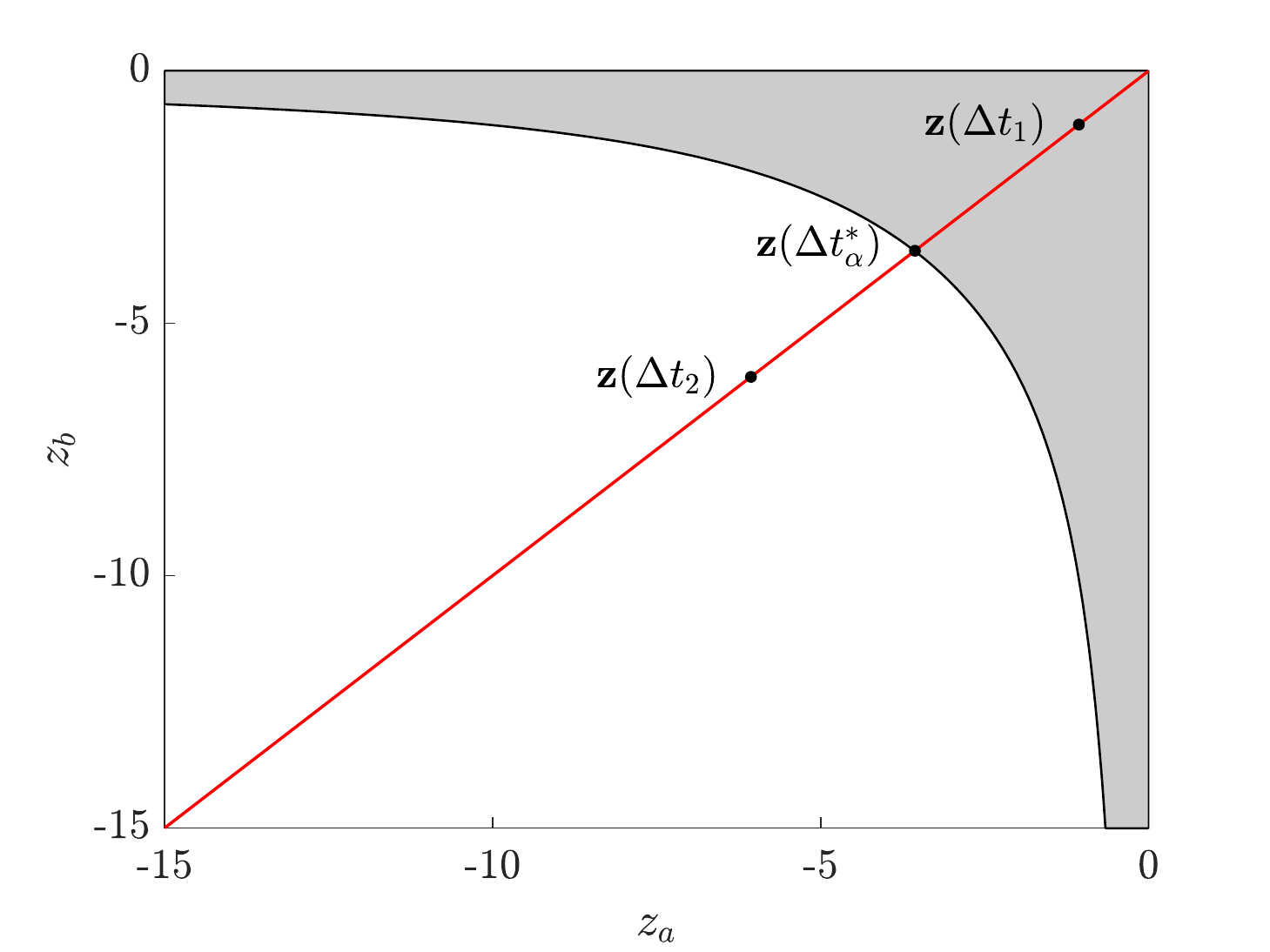}
\subcaption{$\alpha=0.5$ and $\Delta t_\alpha^*=(\sqrt{17}+3)/50\approx 0.14.$ }
\end{subfigure}
\begin{subfigure}[t]{0.495\textwidth}
\includegraphics[width=\textwidth]{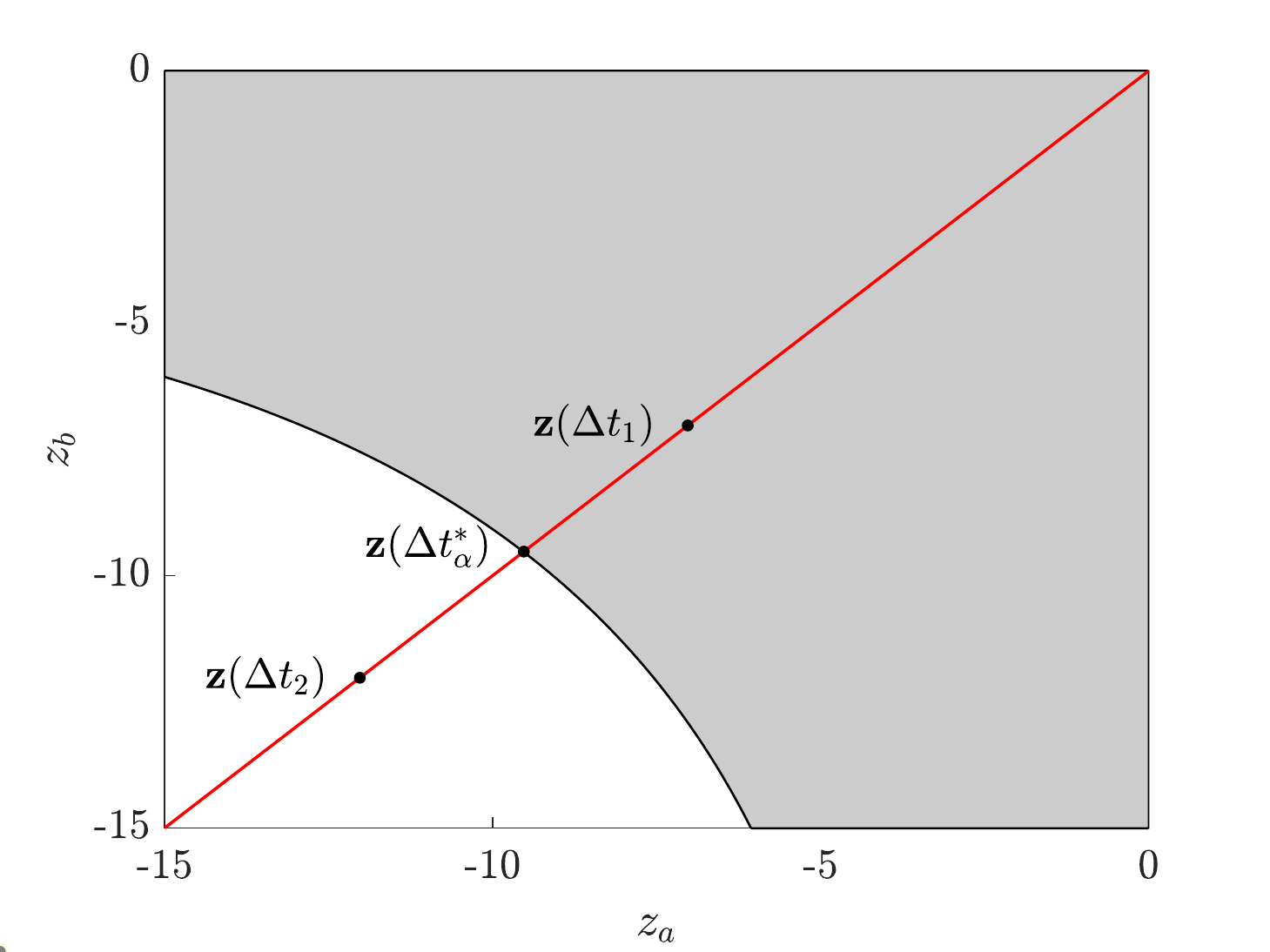}
\subcaption{$\alpha=0.8$ and $\Delta t_\alpha^*=(\sqrt{101}+9)/50\approx 0.38.$ }
\end{subfigure}
\caption{The marked points are $\bm z(\Delta t_2)=-25(\Delta t_\alpha^*+10^{-1})(1,1)^T$, $\bm z(\Delta t_\alpha^*)=-25\Delta t_\alpha^*(1,1)^T$ and $\bm z(\Delta t_1)=-25(\Delta t_\alpha^*-10^{-1})(1,1)^T$, which lie on the red line $z_a=z_b$. The grey areas represent the stability regions for the MPRK22ncs($\alpha$) schemes with $\alpha=0.5$ and $\alpha=0.8$ respectively, see also Figure~\ref{StabilityRegion}.}
\label{Fig:red_line}
\end{figure}

According to Corollary~\ref{cor:unstable}, the MPRK22ncs method can be expected to be stable for $\Delta t_1$ and unstable for $\Delta t_2$. The first expectation is confirmed by Figure~\ref{Fig:MPRK22_Deltat<Deltat*} and additionally MPRK22($\alpha$) is shown to be stable for both time step sizes, which coincides with the statement of Corollary~\ref{cor:MPRK22stabil}. 

\begin{figure}[!h]
\begin{subfigure}[t]{0.495\textwidth}
\includegraphics[width=\textwidth]{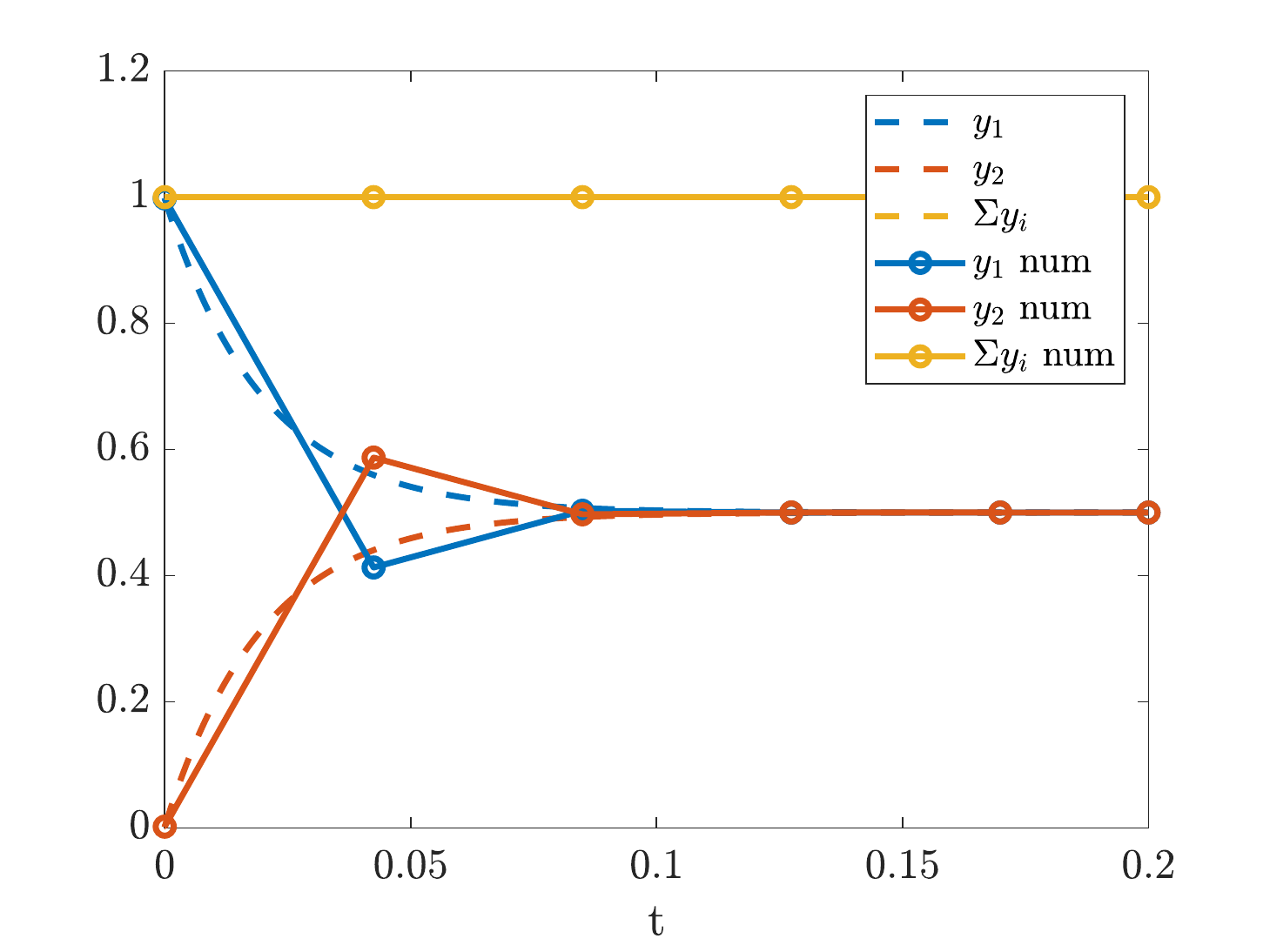}
\subcaption{MPRK22(0.5) with $\Delta t_1 =\Delta t_\alpha^*-10^{-1}\approx 0.04$}
\end{subfigure}
\begin{subfigure}[t]{0.495\textwidth}
\includegraphics[width=\textwidth]{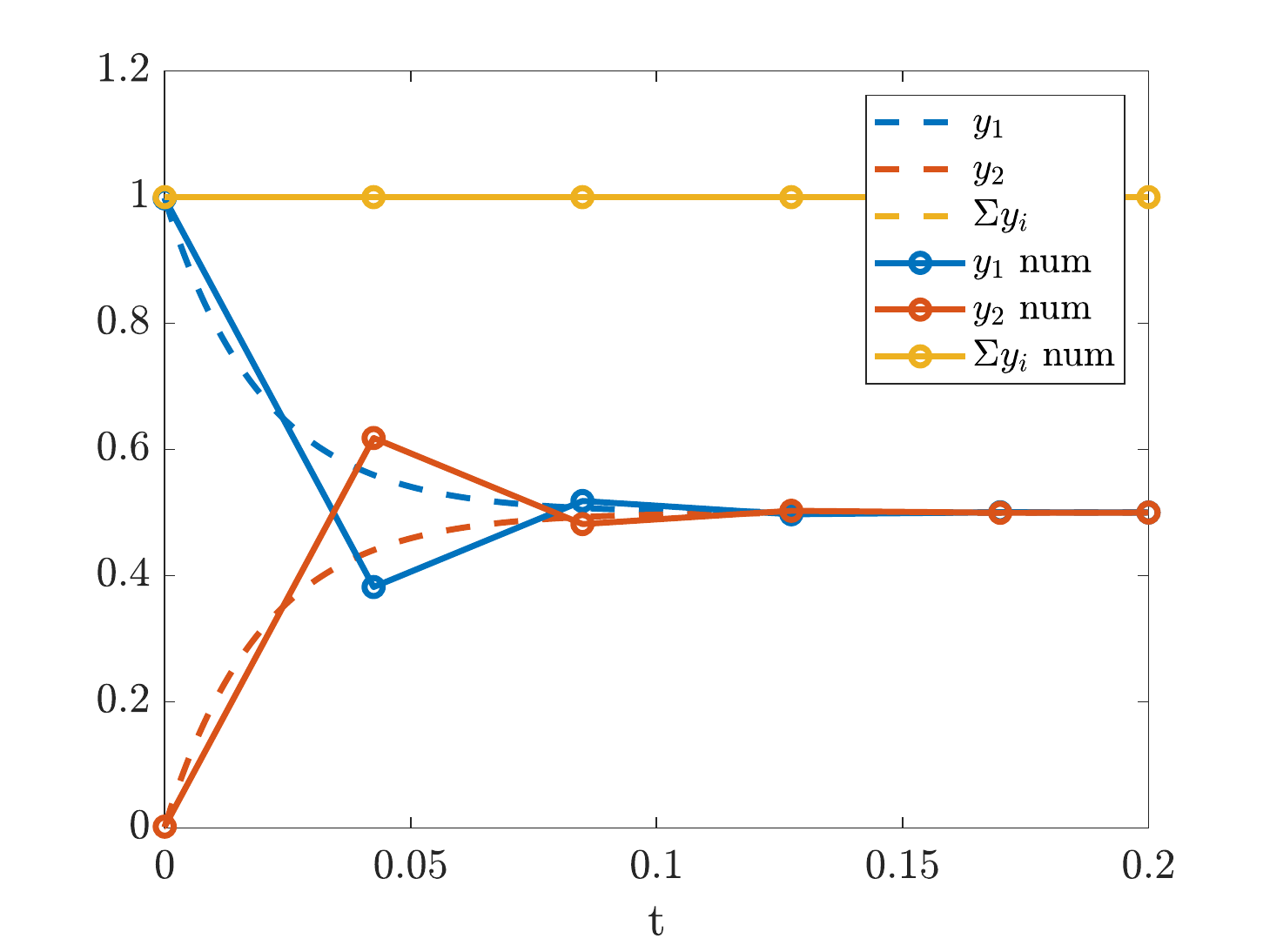}
\subcaption{MPRK22ncs(0.5) with $\Delta t_1 =\Delta t_\alpha^*-10^{-1}\approx 0.04$}
\end{subfigure}\\

\begin{subfigure}[t]{0.495\textwidth}
\includegraphics[width=\textwidth]{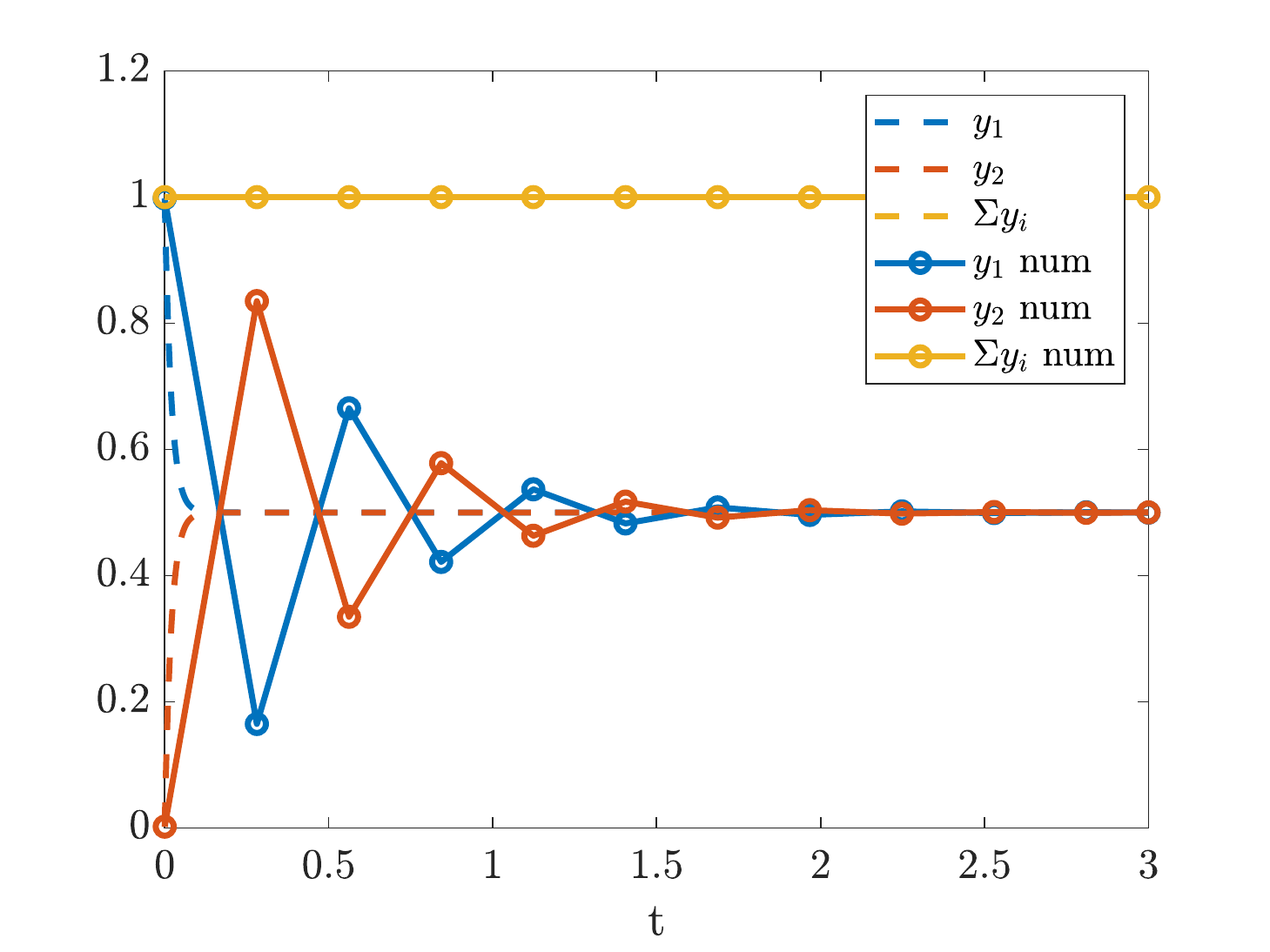}
\subcaption{MPRK22(0.8) with $\Delta t_1 =\Delta t_\alpha^*-10^{-1}\approx 0.28$}
\end{subfigure}
\begin{subfigure}[t]{0.495\textwidth}
\includegraphics[width=\textwidth]{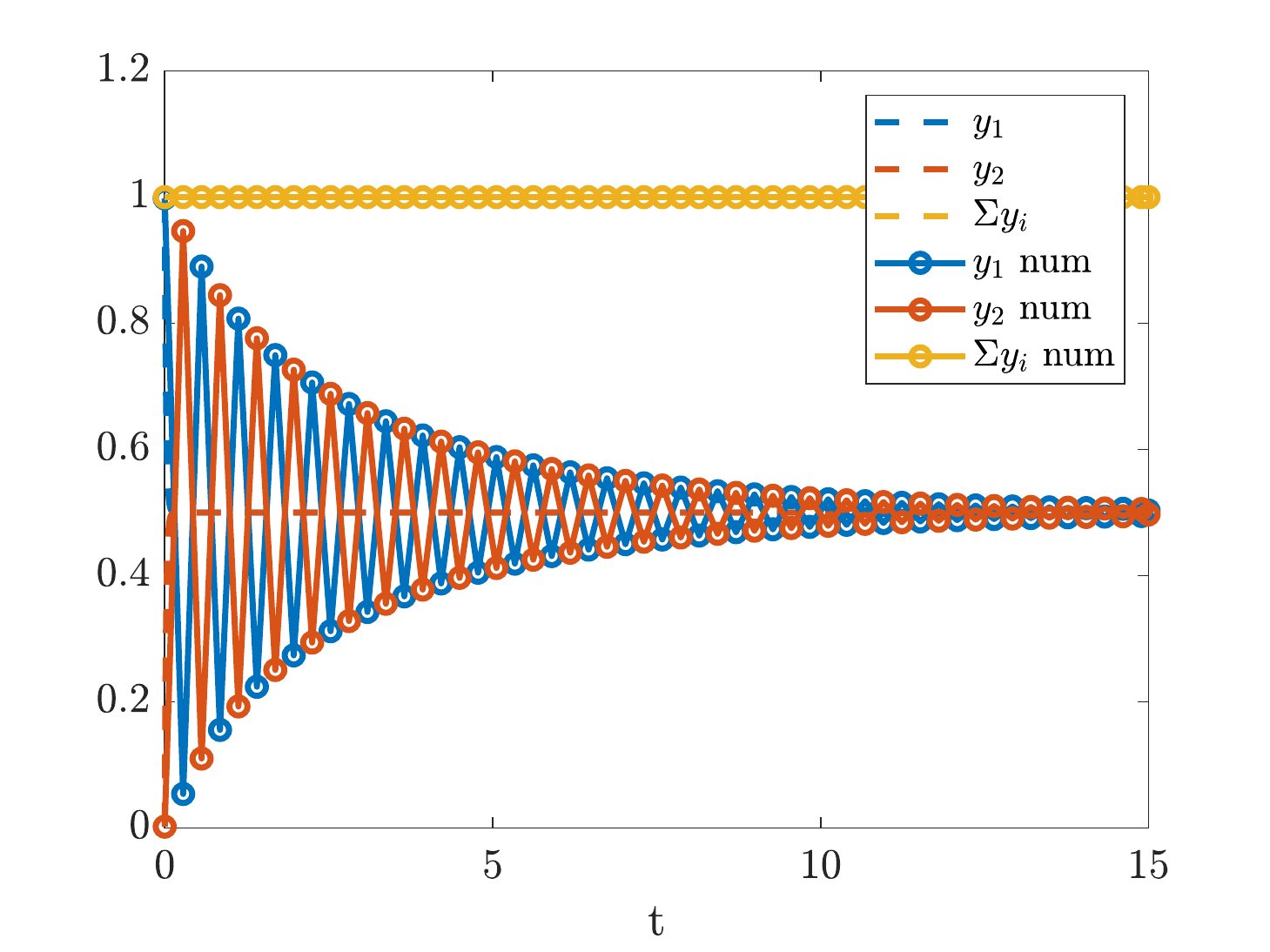}
\subcaption{MPRK22ncs(0.8) with $\Delta t_1 =\Delta t_\alpha^*-10^{-1}\approx 0.28$}
\end{subfigure}
\caption{Numerical approximations of \eqref{test_equation}. The dashed lines indicate the exact solution \eqref{eq:Test_exact}.}
\label{Fig:MPRK22_Deltat<Deltat*}
\end{figure}

In accordance with the presented theory, one can observe the superior stability behavior of MPRK22($\alpha$) in Figure~\ref{Fig:RefSol+Comp} even in the case that the time step size is chosen larger then the critical step size for MPRK22ncs($\alpha$). However, the instability of MPRK22ncs($\alpha$) for $\Delta t_2 = \Delta t_\alpha^* + 10^{-1}$ can only be guessed by the illustration in Figure~\ref{Fig:RefSol+Comp}. To show the divergence of the method more clearly, we modify the initial condition within the initial value problem. Therefore, we consider
\[
\by^0= \by^* +10^{-3}\begin{pmatrix}
1\\-1
\end{pmatrix}= \begin{pmatrix}
0.501\\0.499
\end{pmatrix}
\] 
which is much closer to the steady state than the previously used value. The results shown in Figure~\ref{Fig:InstabZoom} clearly demonstrate the expected divergence from the steady state for $\alpha = 0.5$ as well as $\alpha = 0.8$.

\begin{figure}[!h]
\begin{subfigure}[t]{0.495\textwidth}
\includegraphics[width=\textwidth]{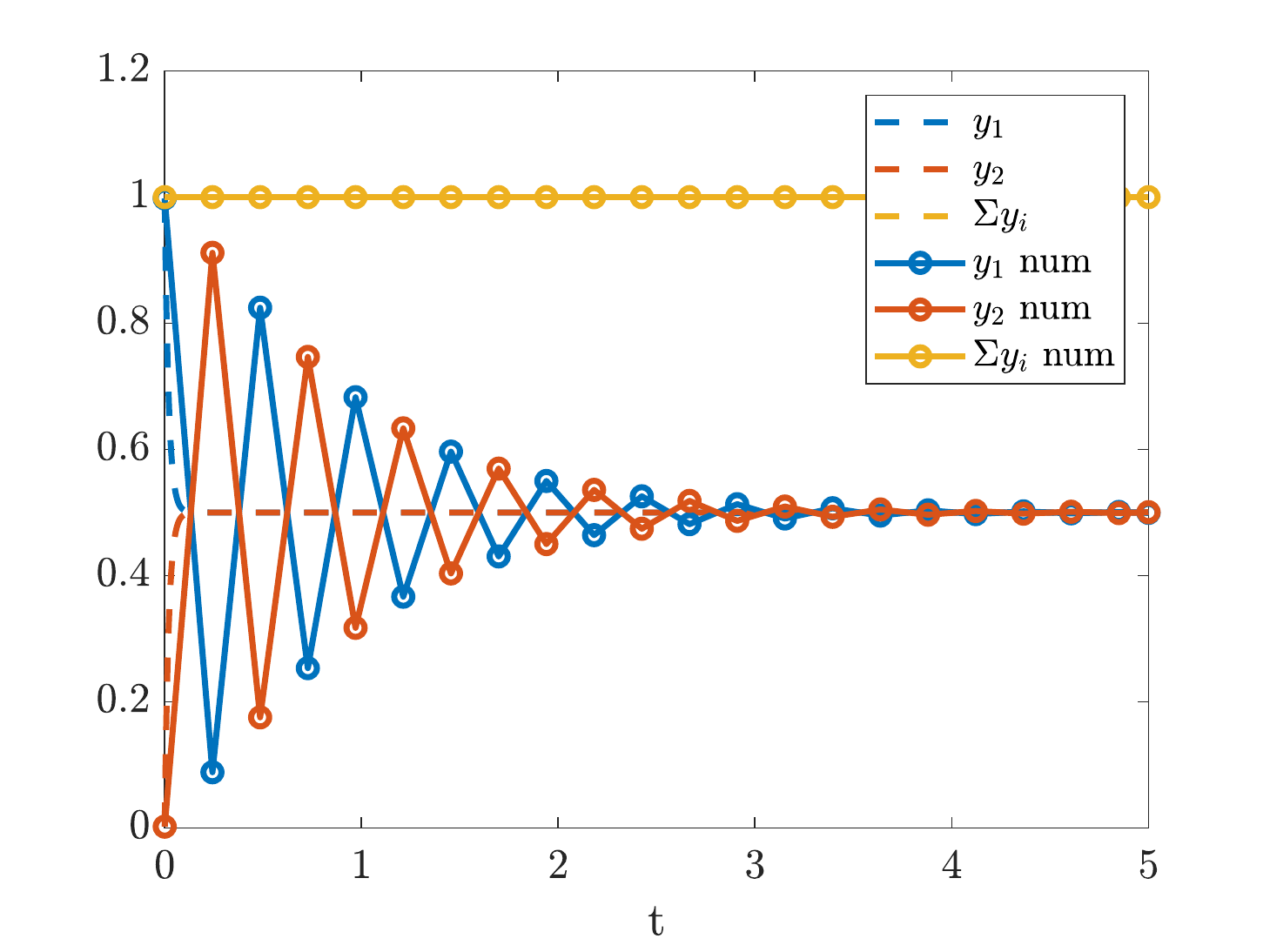}
\subcaption{MPRK22(0.5) with  $\Delta t=\Delta t_\alpha^*+10^{-1}\approx 0.24$}
\end{subfigure}
\begin{subfigure}[t]{0.495\textwidth}
\includegraphics[width=\textwidth]{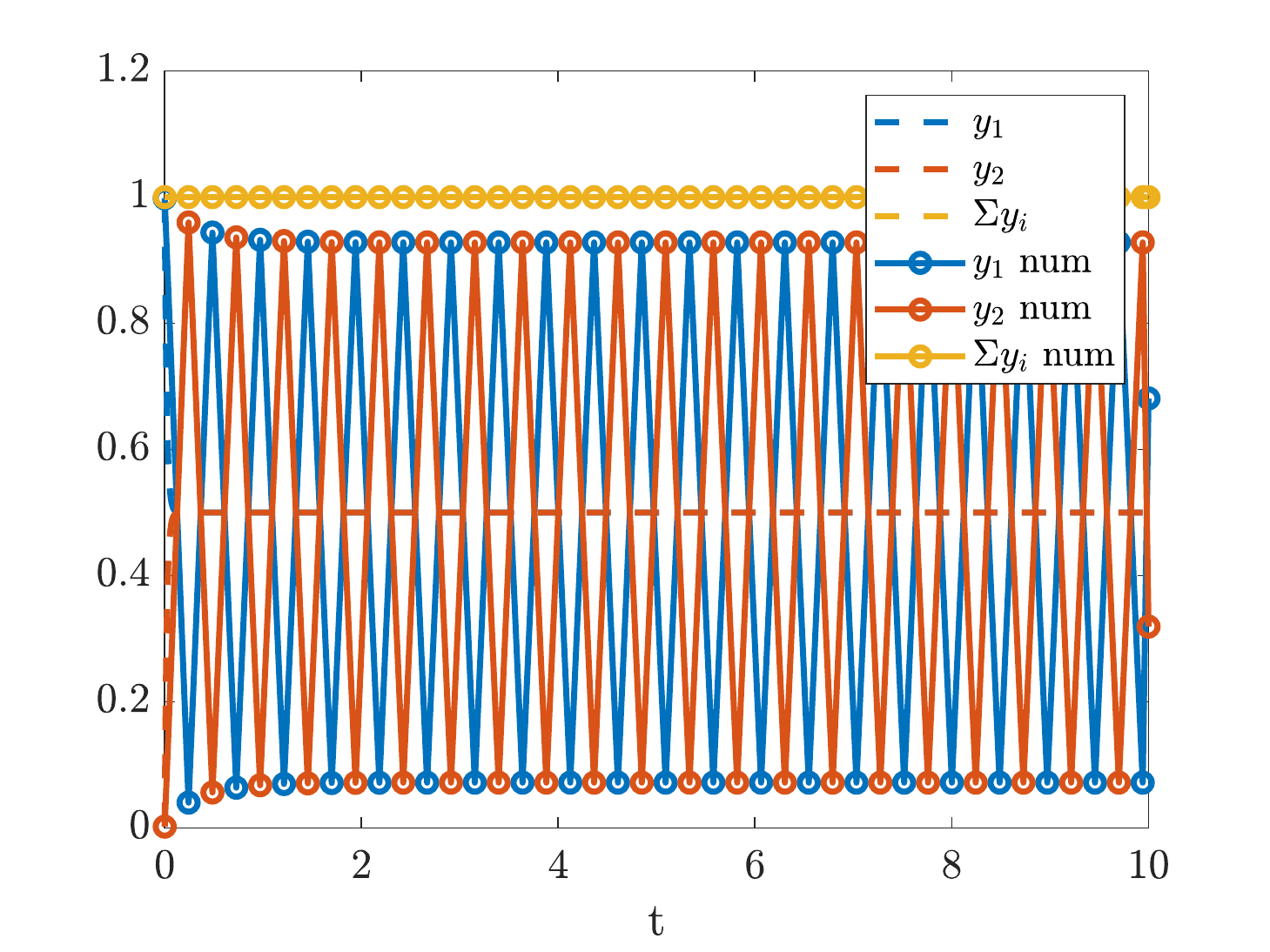}
\subcaption{MPRK22ncs(0.5) with $\Delta t=\Delta t_\alpha^*+10^{-1}\approx 0.24$}
\end{subfigure}\\

\begin{subfigure}[t]{0.495\textwidth}
\includegraphics[width=\textwidth]{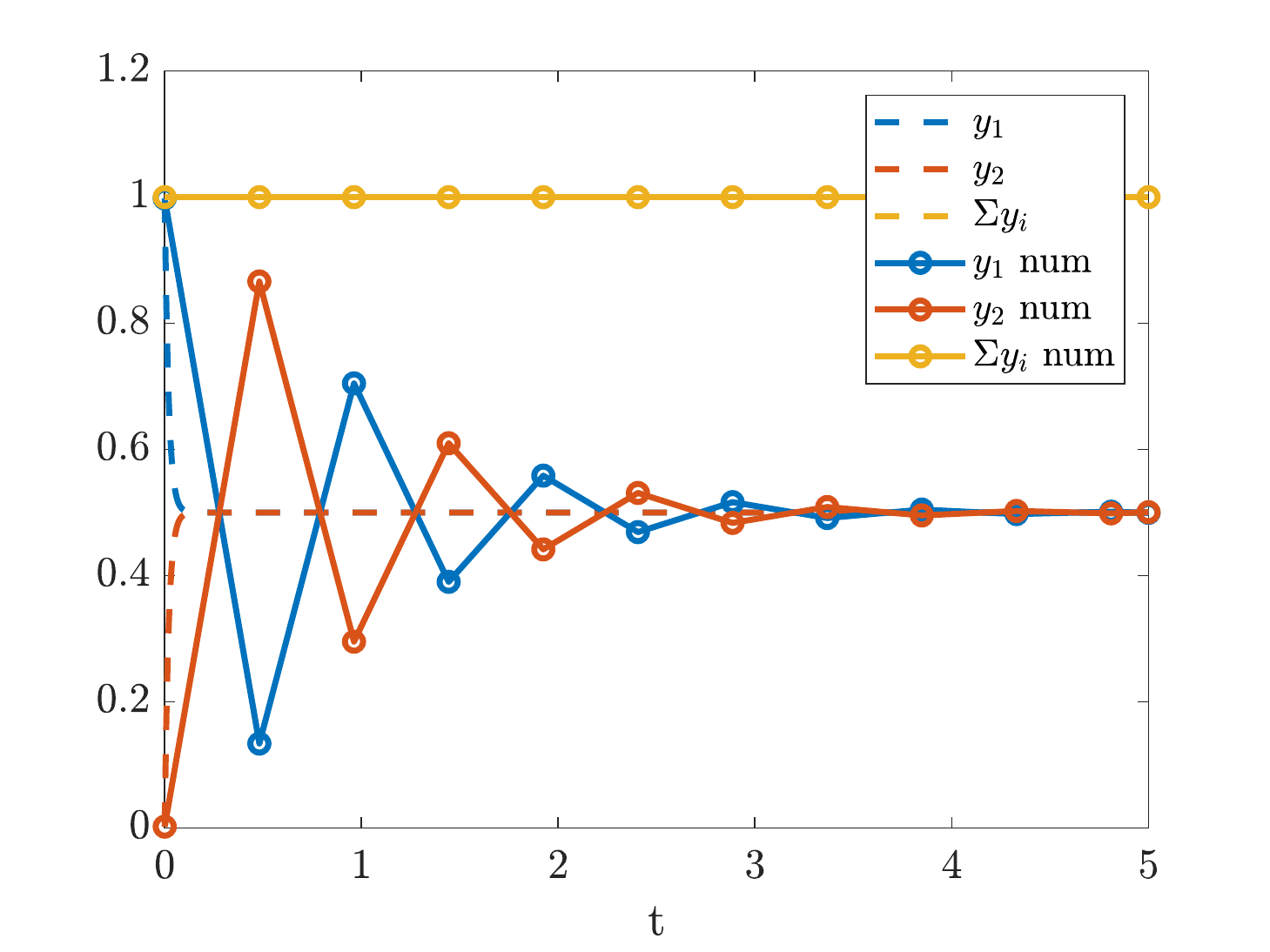}
\subcaption{MPRK22(0.8) with  $\Delta t=\Delta t_\alpha^*+10^{-1}\approx 0.48$}
\end{subfigure}
\begin{subfigure}[t]{0.495\textwidth}
\includegraphics[width=\textwidth]{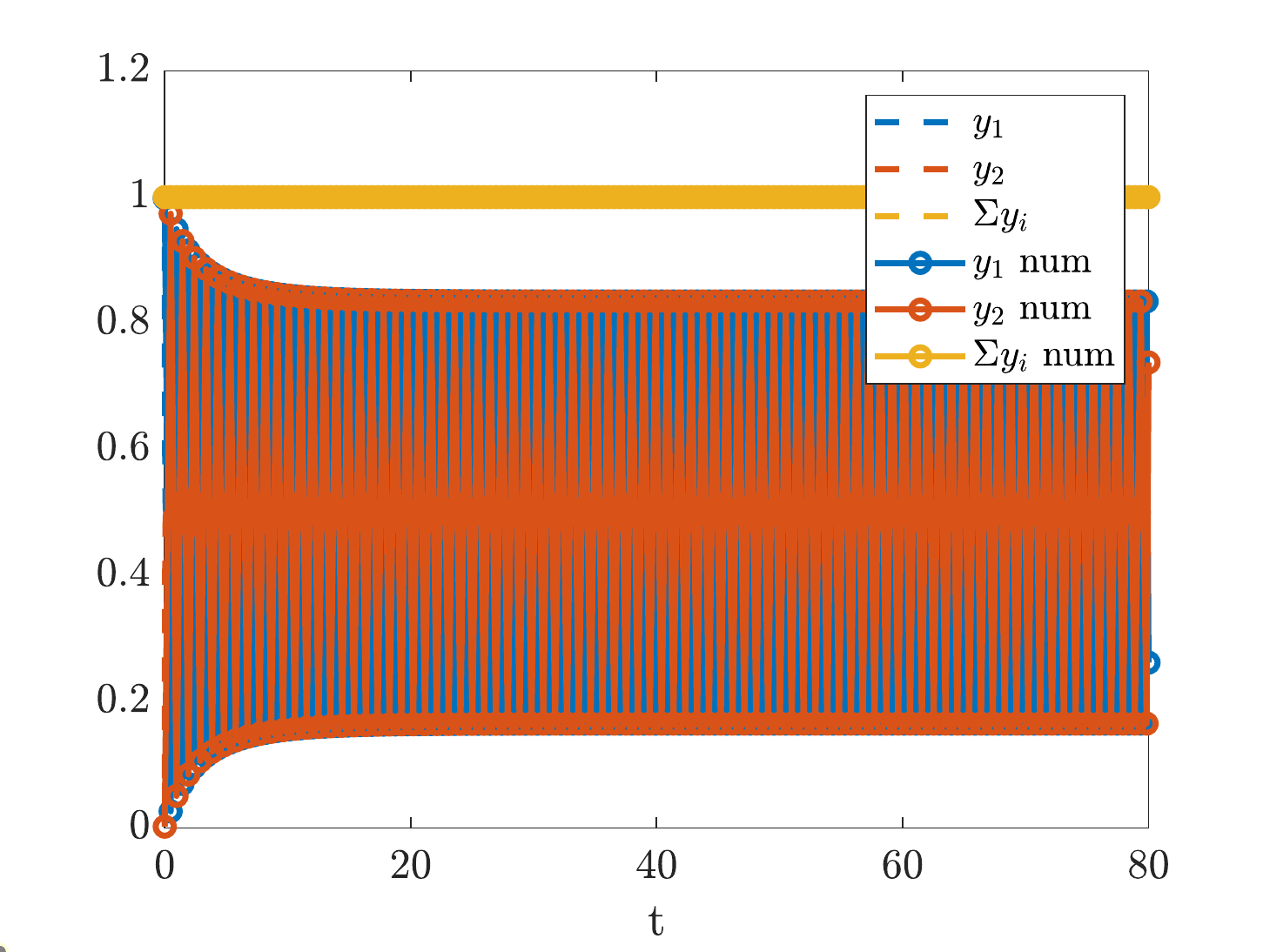}
\subcaption{MPRK22ncs(0.8) with $\Delta t=\Delta t_\alpha^*+10^{-1}\approx 0.48$}
\end{subfigure}
\caption{Numerical approximations of \eqref{test_equation}. The dashed lines indicate the exact solution \eqref{eq:Test_exact}.}\label{Fig:RefSol+Comp}
\end{figure}

\begin{figure}[!h]
\begin{subfigure}[t]{0.495\textwidth}
\includegraphics[width=\textwidth]{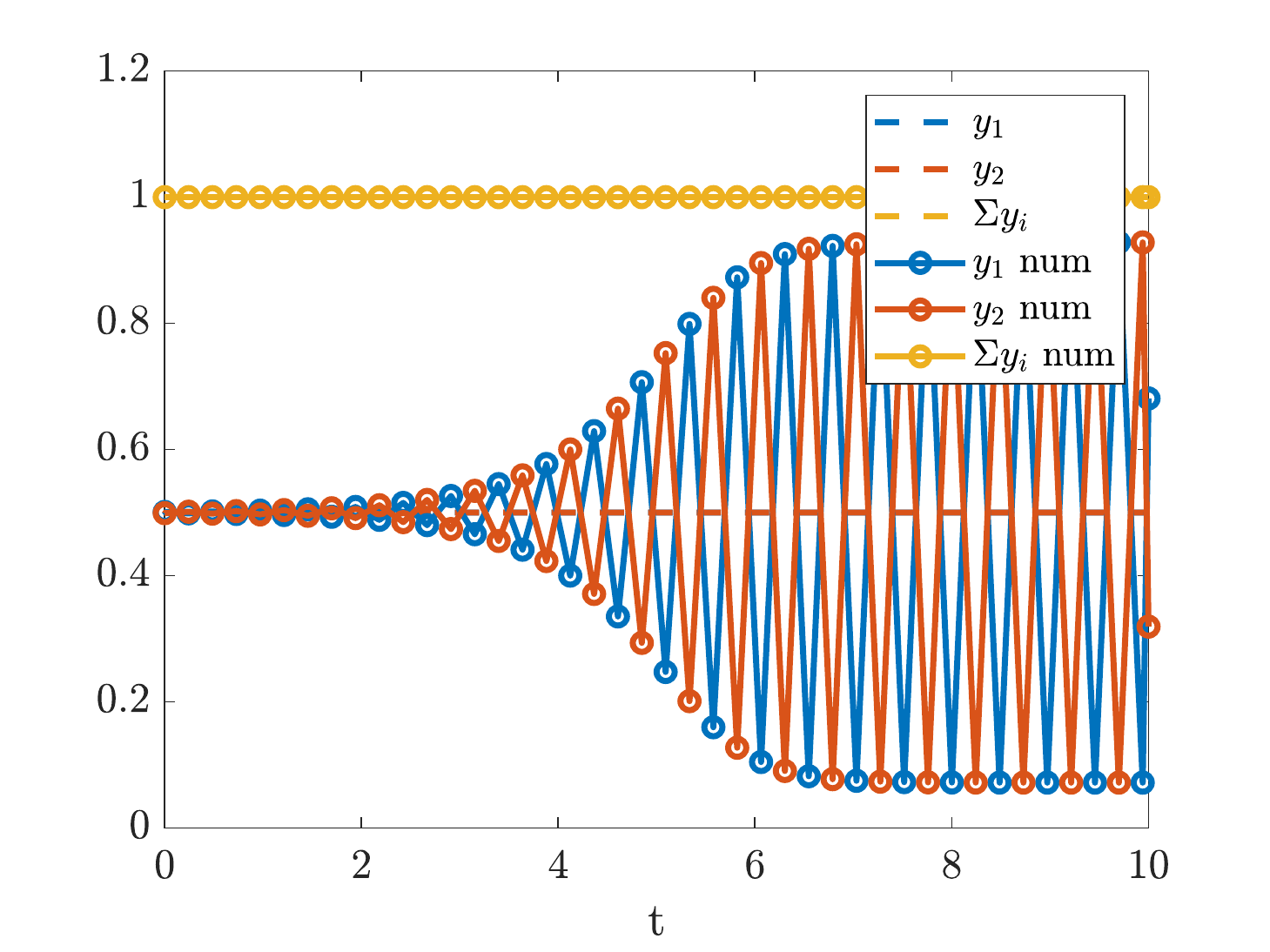}
\subcaption{MPRK22ncs(0.5) with $\Delta t=\Delta t_\alpha^*+10^{-1}\approx 0.24$}
\end{subfigure}
\begin{subfigure}[t]{0.495\textwidth}
\includegraphics[width=\textwidth]{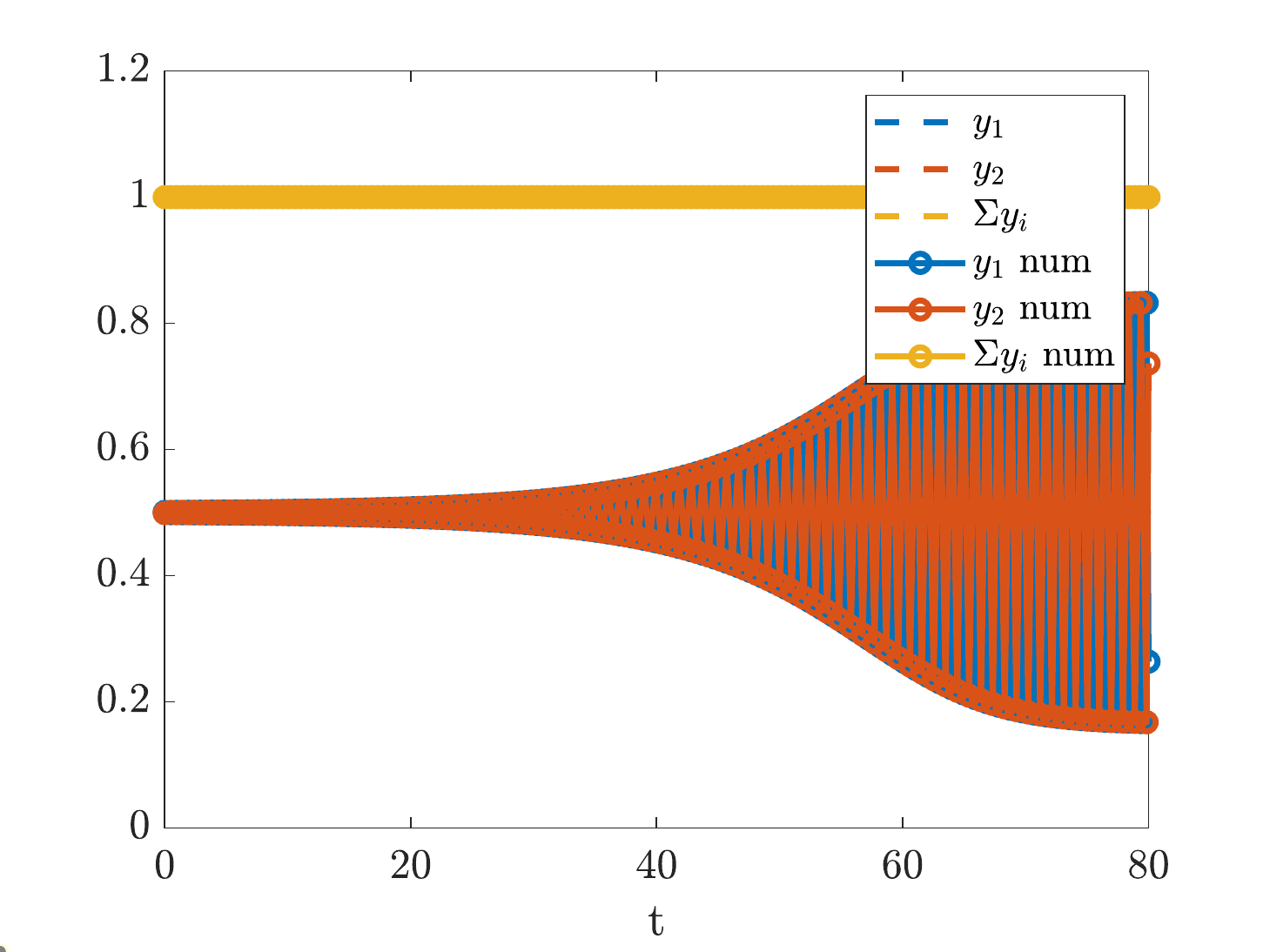}
\subcaption{MPRK22ncs(0.8) with $\Delta t=\Delta t_\alpha^*+10^{-1}\approx 0.48$}
\end{subfigure}
\caption{Numerical approximation of \eqref{PDS_test} with $\by^0=(y_1^*+10^{-3},y_2^*-10^{-3})^T=(0.5+10^{-3},0.5-10^{-3})^T$. The dashed lines indicate the exact solution \eqref{eq:Test_exact}.}\label{Fig:InstabZoom}
\end{figure}

\section{Summary and outlook}
In this paper a stability analysis for general positive and conservative time integration schemes based on the center manifold theory for maps was presented for the first time. The theory shows that even for nonlinear positive and conservative time integrators the investigation of the eigenvalues of the Jacobian is sufficient to analyze stability. This novel theory was used to carry out a first stability analysis of MPRK schemes. Thereby, we discovered that for $\alpha\geq 1$ both MPRK22($\alpha$) and MPRK22ncs($\alpha$) schemes possess stable fixed points irrespective of the chosen time step size $\Delta t$. If $\alpha<1$, these stability properties are maintained by the MPRK22($\alpha)$ schemes,   
whereas the investigation of MPRK22ncs($\alpha$) revealed time step restrictions to ensure stability. We also computed  the corresponding stability regions for MPRK22ncs($\alpha$) schemes.
 
Future research topics include the extension of the statement of Theorem~\ref{Thm_MPRK_stabil} to higher dimensional linear and nonlinear systems as well as the investigation of global stability properties. Also, due to the fact that Theorem~\ref{Thm_MPRK_stabil} is applicable to general positive and conservative schemes a stability analysis of the schemes presented in \cite{KM18Order3,MR3969000,MR3934688,MR4064785,MR4087156,MR4109346, MR4194400} is now possible for the first time.

\section{Acknowledgements}
The author Th.\ Izgin gratefully acknowledges the financial support by the Deutsche Forschungsgemeinschaft
(DFG) through grant ME 1889/10-1.

\bibliographystyle{plain} 
\bibliography{cas-refs}
\end{document}